\documentclass[11pt]{amsart}
\usepackage{a4wide}
\usepackage{version}
\usepackage{amsmath}
\usepackage{amsfonts, amssymb, amsthm}
\usepackage{pstricks}
\usepackage{color,mathrsfs,epsfig}
\usepackage{epsfig}
\usepackage{color}

\DeclareMathAlphabet{\mathpzc}{OT1}{pzc}{m}{it}
\newcounter{comptage}[part]
\newtheorem{lem}[comptage]{Lemma}
\newtheorem{theo}[comptage]{Theorem}
\newtheorem{defin}[comptage]{Definition}

\newtheorem{prop}[comptage]{Proposition}

\newtheorem{preremark}[comptage]{Remark}
\newenvironment{remark}{\begin{preremark}\rm}{\medskip \end{preremark}}

\numberwithin{equation}{section}

\newcommand{\Hh}{\mathcal H}

\newcommand{\R}{\mathbb R}
\newcommand{\N}{\mathbb N}

\newcommand{\diam}{{\rm Diam}}

\newcommand{\dx} {\; \mathrm{d} x}

\newcommand{\dd} {\; \mathrm{d}}

\newcommand{\h}{\mathcal H}

\newcommand{\meanbar}[1]{%
\setbox0 = \hbox{$#1 \int$}
\hbox to 0pt{%
\thinspace
\hskip 0.1\wd0
\raise 0.5\ht0
\hbox{%
\lower 0.5\dp0
\hbox{\rule{0.8\wd0}{2\linethickness}}
}%
\hss
}%
}

\author{Antoine Lemenant, Emmanouil Milakis and Laura V. Spinolo}

\title[Spectral Stability Estimates in rough domains]{Spectral Stability Estimates for the Dirichlet and Neumann Laplacian in rough domains.}

\begin{document}

\maketitle

\begin{abstract} In this paper we establish new quantitative stability estimates with respect to domain perturbations for all the eigenvalues of both the Neumann and the Dirichlet Laplacian. Our main results follow from an abstract lemma stating that it is actually sufficient to provide an estimate on suitable projection operators. Whereas this lemma could be applied under different regularity assumptions on the domain, here we use it to estimate the spectrum in Lipschitz and in so-called Reifenberg-flat domains. Our argument also relies on suitable extension techniques and on an estimate on the decay of the eigenfunctions at the boundary which could be interpreted as a boundary regularity result. 

\end{abstract}



\tableofcontents

\section{Introduction}
\subsection{Main results}
In this paper we deal with the eigenvalues of the Laplace operator for the Dirichlet and the Neumann problems in rough domains. Let $\Omega \subseteq \R^N$ be an open and bounded set. If $\partial \Omega$ satisfies suitable mild regularity assumptions, then we can apply the classical results concerning the spectrum of compact operators and infer that both the Dirichlet and the Neumann problems admit a sequence of nonnegative eigenvalues, which we denote by 
$$
    0 < \lambda_1 (\Omega) \leq \lambda_2 (\Omega) \leq \dots \leq \lambda_k  (\Omega)
\leq \dots \uparrow + \infty
$$
and 
$$
    0 = \mu_1 (\Omega) \leq \mu_2 (\Omega) \leq \dots \leq \mu_k  (\Omega) \leq \dots \uparrow + \infty, 
$$
respectively. Each eigenvalue is counted according to its multiplicity. 

The problem of investigating the way that the eigenvalues $\lambda_k$ and $\mu_k$ depend on the domain $\Omega$ has been widely studied. We refer to the books by Bucur and Buttazzo~\cite{BucurButtazzo} and by Henrot~\cite{Henrot} and to references therein for an extended discussion. See also the expository work by Hale~\cite{Hale}. In the present paper we establish new stability estimates concerning the dependence of the eigenvalues on domain perturbations. The most relevant features of the paper are the following.

First, our approach applies to both the Dirichlet and the Neumann problem, while many of the previous results were concerned with the Dirichlet problem only. The key argument in the present paper is based on an abstract lemma (Lemma~\ref{l:ab} in Section~\ref{s:estpro}) which is sufficiently general to apply to both Dirichlet and Neumann problems. Lemma~\ref{l:ab} has an elementary proof which uses ideas due to Birkhoff, de Boor, Swartz and Wendroff~\cite{BBSW}. Although in this paper we choose to mainly focus on domains satisfying a specific regularity condition, first introduced by E. R. Reifenberg~\cite{r},  Lemma~\ref{l:ab} can be applied to other classes of domains. As an example, we consider the case of Lipschitz domains, see Theorem~\ref{mainLip}.

Second, we impose very weak regularity conditions on the domains $\Omega_a$ and $\Omega_b$. The exact definition of the regularity assumptions we impose is given later, here we just mention that Reifenberg flatness is a property weaker than Lipschitz continuity and that Reifenberg-flat domains are relevant for the study of minimal surfaces~\cite{r} and of other problems, see Toro~\cite{t} for an overview.  

Third, in the present paper we establish quantitative estimates, while much of the analysis discussed in \cite{BucurButtazzo,Henrot} aimed at proving existence and convergence results.  More precisely, we will obtain estimates of the following type:  
\begin{equation}
\label{e:intro:me}
\begin{split}
& \qquad \qquad  |\lambda_k(\Omega_a)-\lambda_k(\Omega_b)| \leq C d_H(\Omega_a^c , \Omega_b^c)^\alpha \\
& |\mu_k(\Omega_a)-\mu_k(\Omega_b)| \leq C \max(d_H(\Omega_a^c , \Omega_b^c),d_H( \Omega_a ,  \Omega_b ))^\alpha. \\
 \end{split}
\end{equation}
In the previous expressions, $\Omega^c:=\R^N \setminus \Omega$  stands for the complement of the set $\Omega$ 
and $d_H$ for the Hausdorff distance, namely 
\begin{equation}
\label{e:h}
    d_H(  X ,  Y ) : = \max \big\{ \sup_{x \in X } d(x,  Y),
   \sup_{y \in  Y}d(y,  X) \big\}.
\end{equation}
In the following we will discuss both the admissible values for the exponent $\alpha$ and the reasons why the quantity $d_H( \Omega_a ,  \Omega_b )$ only appears in the stability estimate for the Neumann eigenvalues.

Note that quantitative stability results were established in a series of papers by Burenkov, Lamberti, Lanza De Cristoforis and collaborators (see~\cite{BLL} for an overview). However, the regularity assumptions we impose on the domains are different than those in~\cite{BLL} and, moreover, our approach relies on different techniques. Indeed, the analysis in~\cite{BLL} is based on the notion of  transition operators, while as mentioned before our argument combines real analysis techniques with an abstract lemma whose proof is based on elementary tools.  

We also refer to a very recent work by Colbois, Girouard and Iversen \cite{cgi} for other quantitative stability results concerning the Dirichlet problem.

As a final remark, we point out that Lemma 19 in~\cite{GeomPart} ensures that, given two sufficiently close Reifenberg-flat domains $\Omega_a, \Omega_b \subseteq \R^N$, the Hausdorff distance is controlled by the Lebesgue measure of the symmetric difference, more precisely 
$$
     d_H( \Omega_a^c,  \Omega_b^c)  \leq C |\Omega_a \triangle \Omega_b|^{\frac{1}{N}}, \quad  d_H(  \Omega_a,  \Omega_b  )  \leq C |\Omega_a \triangle \Omega_b|^{\frac{1}{N}}, 
$$
where the constant $C$ only depends on the dimension $N$ and on a regularity parameter of the domains. 
Hence, an immediate consequence of~\eqref{e:intro:me} is that the corresponding estimates holds in terms of $|\Omega_a \triangle \Omega_b|^{\frac{1}{N}}$. 
Before introducing our results, we specify the exact regularity assumptions we impose on the sets $\Omega_a$ and $\Omega_b$. 

\begin{defin}\label{defreif} Let $\varepsilon, r_0$ be two real numbers satisfying $0 < \varepsilon<1/2$ and $r_0 >0$. An $(\varepsilon,r_0)$-Reifenberg-flat domain $\Omega \subseteq \R^N$ is a nonempty open set satisfying the following two conditions:

\noindent i) for every $x \in \partial \Omega$ and for every $r\leq r_0$,  there is a hyperplane $P(x,r)$ containing $x$ which satisfies
\begin{eqnarray}
\frac{1}{r}d_H( \partial \Omega \cap B(x,r), P(x,r)\cap B(x,r))\leq \varepsilon. \label{reif}
\end{eqnarray}

\noindent ii) For every $x \in \partial \Omega$,  one of the connected component of  
$$B(x,r_0)\cap \big\{x : \;  dist(x,P(x,r_0))\geq 2\varepsilon r_0\big\}$$ 
is contained in   $\Omega$ and the other one is contained in $\R^N \setminus \Omega$.
\end{defin}
A direct consequence of the definition is that if $\varepsilon_1 < \varepsilon_2$, then any $(\varepsilon_1,r_0)$-Reifenberg-flat domain is also an $(\varepsilon_2,r_0)$-Reifenberg-flat domain. Note also that, heuristically speaking, condition $i)$ ensures that the boundary is well approximated by  hyperplanes at every small scale, while condition $ii)$ is a separating requirement equivalent to those in~\cite{hm3,hm2, hm1,mt}. The notion of Reifenberg flatness is strictly weaker than Lipschtiz continuity and appears in many areas like free boundary regularity problems and geometric measure theory (see \cite{d1,d2,ddpt,hm3,hm2,hm1,l1,l2,lm2,lm,mt,r,t,hm4} and the references therein).   

We now state our main result concerning the Dirichlet problem. We denote by $\mathcal H^{N-1}$ the Hausdorff $(N-1)$-dimensional measure.

\begin{theo}[Dirichlet Problem]\label{main} Let $B_0, D \subseteq \R^N$ be two given balls satisfying $B_0 \subseteq D$ and denote by  $(\gamma_n)_{n\in \N}$ the spectrum of the Dirichlet Laplacian in $B_0$ and by $R$ the radius of $D$. 

For any $\alpha\in ]0,1[$ there is $\varepsilon=\varepsilon(\alpha)$ such that the following holds. For any $n \in \N$, ${r_0>0}$ and $L_0>0$ there are constants $\delta_0=\delta_0(\gamma_n, n, \alpha, r_0, N,L_0, R )$ and $C=C( \alpha, r_0, N)$ such that whenever $\Omega_a$ and $\Omega_b$ are two $(\varepsilon,r_0)$-Reifenberg-flat domains in $\R^N$ such that
\begin{itemize}
\item $B_0 \subseteq \Omega_a \cap  \Omega_b$ and $\Omega_a \cup \Omega_b \subseteq D$;
\item  $L:=\max(\Hh^{N-1}(\partial \Omega_a),\Hh^{N-1}(\partial \Omega_b))\leq L_0;$
\item $d_H(\Omega_a^c, \Omega_b^c )\leq \delta_0,$
\end{itemize}
then 
\begin{equation}
\label{e:dir}
|\lambda_n ^a -\lambda_n^b|\leq C n \gamma_n(1+\gamma_n^{\frac{N}{2}})L d_H(\Omega_a^c, \Omega_b^c)^\alpha,
\end{equation}
where $\{ \lambda_n^a \}$ and $\{ \lambda_n ^b \}$ denote the sequences of eigenvalues of the Dirichlet Laplacian in $\Omega_a$ and $\Omega_b$, respectively. 
\end{theo}
In Section \ref{s:outline} we provide an outline of the proof of Theorem \ref{main}, here we make some remarks. 

First, Theorem \ref{main} is an extension of the main theorem in Lemenant Milakis \cite{lm2}, which establishes a similar result valid for the first eigenvalue only. The extension to the whole spectrum is not trivial and relies on the abstract result provided by Lemma~\ref{l:ab}.

Second, we are confident that our method could be extended to wider classes of linear elliptic equations with no substantial changes.

Third, we only define the ball $B_0$ and the sequence 
$\{ \gamma_n \}$ to simplify the statement of Theorem~\ref{main}. Indeed, we could have provided a sharper statement by letting the constant $C$ and $\delta_0$ directly depend on $\lambda_n^a$ and $\lambda_n^b$, but for sake of simplicity we decided to use the monotonicity of eigenvalues with respect to domain inclusion, which ensures that
$ \max\{ \lambda_n^a,\lambda_n^b \} \leq \gamma_n.$  Note  that in the Neumann case this monotonicity property fails and hence in the statement of  Theorem \ref{main2} the constants  explicitly depend on the eigenvalues. Also, note that $\gamma_n$ only depends on the inner radius of the domains $\Omega_a$ and $\Omega_b$. 

Finally, we make some remarks concerning the exponent $\alpha$ in \eqref{e:dir}. If an inequality like~\eqref{e:dir} holds, then $\alpha \leq 1$. This can be seen by recalling that the Sobolev-Poincar\'e constant  of the ball of radius $R$ in $\R^N$ is proportional to $R^2$. 
Note, however, that in~\eqref{e:dir} we require $\alpha<1$ and hence we do not reach the optimal exponent $\alpha=1$. Also, the regularity parameter $\varepsilon$ depends on $\alpha$, and the closer $\alpha$ is to $1$, the smaller is $\varepsilon$. Those restrictions are due to technical reasons connected to the fact that we are imposing fairly weak regularity assumptions on the domains.

However, if we restrict to Lipschitz domains we obtain sharper estimates since $\alpha$ can attain the optimal value $\alpha =1$. Also, in the case of Lipschitz domains no restriction is imposed on the regularity parameter, i.e. we do not need to require that the domains are ``sufficiently flat" as in the Reifenberg-flat case. 

Before stating the precise result, we have to introduce the following definition. 

\begin{defin} Let $\Omega \subseteq \R^N$ be an open set, then $\Omega$ satisfies a uniform $(\rho,\theta)$-cone condition if for any $x \in \partial \Omega$ there is a unit vector $\nu \in  \R^{N-1}$, possibly depending on $x$, such that
$$B(x, 3\rho)\cap \Omega - C_{\rho,\theta}(\nu) \subseteq \Omega \quad \text{ and }\quad B(x,3 \rho)\setminus \R^N  + C_{\rho,\theta}(\nu) \subseteq \Omega \setminus \R^N,$$
where $B(x,3\rho)$ denotes the ball centered at $x$ with radius $3 \rho$ and $C_{\rho,\theta}(\nu)$ is the cone with height $\rho>0$ and opening $\theta\in ]0,\pi]$, 
$$C_{\rho,\theta}(\nu):=\{ h \in \R^N: \; h\cdot \nu > |h| \cos \theta\} \cap B(\vec 0, \rho).$$ 
\end{defin}
We now state the stability result for Lipschitz domains.
\begin{theo}[Dirichlet Problem in Lipschitz domains]
\label{mainLip} 
Let $B_0, D \subseteq \R^N$ be two given balls satisfying $B_0 \subseteq D$ and denote by  $(\gamma_n)_{n\in \N}$ the spectrum of the Dirichlet Laplacian in $B_0$ and by $R$ the radius of $D$. 

For any $\rho>0, \; \theta \in  ]0,\pi]$ there are constants $C=C(\rho,\theta,n, N,\gamma_n, R)$ and $\delta_0 = \delta_0(\rho,\theta, n, N,\gamma_n, R)$ such that the following holds. Let $\Omega_a$ and $\Omega_b$ be two open sets satisfying a $(\rho, \theta)$-cone condition and the following properties: 
\begin{itemize}
\item $B_0 \subseteq \Omega_a \cap \Omega_b$ and $\Omega_a \cup \Omega_b   \subseteq D$;
\item  
$
    \delta: =d_H (\Omega_a^c, \Omega_b^c) \leq \delta_0.
$
\end{itemize}
Then
$$
|\lambda_n^a -\lambda_n^b |\leq C \delta.
$$
\end{theo}
The proof of Theorem \ref{mainLip} combines the above-mentioned abstract result (Lemma~\ref{l:ab}) with stability estimates for solutions of elliptic equations due to Savar\'e and Schimperna~\cite{SavareSchimperna}. 

We now state the stability result concerning the Neumann problem. 
\begin{theo}[Neumann Problem]
\label{main2}
For any $\alpha \in ]0, 1[$ there is $\varepsilon = \varepsilon (\alpha)$ such that the following holds. Let $\Omega_a$ and $\Omega_b$ be two bounded, connected, $(\varepsilon, r_0)$-Reifenberg-flat domains in $\R^N$ such that 
\begin{itemize}
\item $L:=\max(\Hh^{N-1}(\partial \Omega_a),\Hh^{N-1}(\partial \Omega_b))\leq L_0;$
\item both $\Omega_a$ and $\Omega_b$ are contained in the ball $D$, which has radius $R$. 
\end{itemize} 
 Let $\mu_n^a$ and $\mu_n ^b$ be the corresponding sequences of Neumann Laplacian eigenvalues and denote by $\mu_n^*:=\max \{  \mu_n^a, \mu_n ^b \}$. For any $n \in \N$, there are constants $\delta_0=\delta_0(\mu^*_n, \alpha, r_0, n, N, R, L_0 )$ and  $C=C(N, r_0, \alpha, R)$ such that if 
$$\max \big\{ d_H( \Omega_a^c,  \Omega_b^c ),d_H( \Omega_a,  \Omega_b)\big\}\leq \delta_0,$$
then 
\begin{eqnarray}
|\mu_n ^a -\mu_n^b|\leq C n \big(1+\sqrt{\mu_n^*} \big)^{2 \gamma(N)+2} L(\max(d_H( \Omega_a^c,  \Omega_b^c),d_H( \Omega_a,  \Omega_b))^\alpha, \label{estimaNEUM}
\end{eqnarray}
where $\gamma(N)=\max\left\{ \displaystyle{\frac{N}{2},\frac{2}{N-1} }\right\}$.
\end{theo}
Five remarks are here in order:
\begin{enumerate}
\item In the statement of Theorem~\ref{main2} we again impose that $\varepsilon$ is sufficiently small. Note, however, that in the case of the Neumann problem there are counterexamples showing that, if the boundary is not flat enough, then the stability result may fail, see the book by Courant and Hilbert~\cite[p.420]{CourantHilbert}. Also, as in the Dirichlet case we do not reach the optimal exponent $\alpha=1$ but we are arbitrary close to it, up to a small enough $\varepsilon$. 
\item In the statement of Theorem~\ref{main2} we imposed that the domains $\Omega_a$ and $\Omega_b$ are both connected. We did so to simplify the exposition. However, by relying on the analysis in~\cite[Section4]{GeomPart}, one can remove the connectedness assumption: the price one has to pay is that the threshold $\varepsilon$ for the stability estimate to hold not only depends on $\alpha$, but also on the dimension $N$. 
\item We point out that the Neumann problem poses much more severe challenges than the Dirichlet problem.  For instance, given an open set $\Omega \subseteq \R^N$, we can always extend a function in $H^1_0 (\Omega)$ to obtain a Sobolev function defined on the whole $\R^N$ by setting $u=0$  outside $\Omega$.  This leads to the classical monotonicity property of the eigenvalues with respect to domain inclusions.  Conversely, extending a function in $H^1 (\Omega)$ outside $\Omega$ is not trivial and one needs some regularity assumption on $\Omega$. This issue is investigated in the companion paper~\cite{GeomPart}, see Corollary 6 in there for the precise statement concerning the extension of Sobolev functions defined in Reifenberg-flat domains. Moreover, the monotonicity with respect to domain inclusion is in general false in the case of the Neumann problem. This is the reason why in the statement of Theorem~\ref{main2} we have introduced the quantity $\mu^{\ast}_n$. 
\item A consequence of the fact that the Neumann problem is more difficult to handle than the Dirichlet problem is the following. The stability estimate in Theorems~\ref{main} and~\ref{mainLip} (Dirichlet problem) only involves $d_H(\Omega_a^c, \Omega_b^c)$. Conversely, the stability estimate in Theorem~\ref{main2} (Neumann problem) involves both $d_H(\Omega_a, \Omega_b)$ and $d_H(\Omega_a^c, \Omega_b^c)$. From  the heuristic viewpoint, the reason of this discrepancy is the following.

In the case of the Dirichlet problem, a key point in the proof is constructing a function $\tilde u \in H^1_0 (\Omega_a)$ which is ``sufficiently close" to a given a function $u \in H^1_0 (\Omega_b)$. To achieve this, what we do is loosely speaking we modify $u$ in such a way that it takes the value $0$ in the region $\Omega_b \setminus \Omega_a$. It turns out that, to control the difference between the function $u$ and the function $\tilde u$ constructed in this way, we only need to control $d_H(\Omega_a^c, \Omega_b^c)$. This a consequence of the fact that we do not need to control the behavior of $\tilde u$ on $\Omega_a \setminus \Omega_b$ because by construction $\tilde u$ takes the value zero there. The details of this construction are provided in the proof of Proposition~\ref{estimproj}. 

On the other hand, in the case of the Neumann problem one has to construct a function $\tilde u \in H^1(\Omega_a)$ which is ``sufficiently close" to a given function $u \in H^1 (\Omega_b)$. It turns out that in this case, to control the difference $\tilde u - u$, one has to control both $d_H(\Omega_a, \Omega_b)$ and $d_H(\Omega_a^c, \Omega_b^c)$. This a consequence of the fact that one has to control the behavior of the functions on both $\Omega_a \setminus \Omega_b$ and $\Omega_b \setminus \Omega_a$. The details of this construction are provided in the proof of Proposition~\ref{estimproj2}.
\item Finally, we underline that one could determine how the constants in the statements of Theorems~\ref{main},~\ref{mainLip} and~\ref{main2} explicitly depend on the various parameters by examining the proofs. 
\end{enumerate}
\subsection{Outline}
\label{s:outline}
The proofs of  Theorems~\ref{main} and~\ref{main2} both rely on the following steps.
\begin{enumerate}
\item[\rm{(i)}] First, we establish an abstract lemma (Lemma~\ref{l:ab} in Section~\ref{ss:lemma}) which reduces the problem of estimating the difference between the eigenvalues to the problem of estimating the difference between the eigenfunctions and their projections on suitable subspaces. For instance, in the case of the Dirichlet problem it is enough to establish appropriate estimates on the norm of the orthogonal projection from $H^1_0(\Omega_b)$ onto $H^1_0(\Omega_a)$.

\item[\rm{(ii)}] Next, we estimate the difference between the eigenfunctions and their projections. For both the Dirichlet and the Neumann cases we employ the same strategy (which was already used in the previous work \cite{lm2} by the first two authors). Namely, we employ a covering argument which reduces the problem of estimating the eigenfunction projection to the problem of controlling the energy of the eigenfunction close to the boundary. For instance, in the Dirichlet case we show that, given the eigenfunction $u \in H^1_0(\Omega_b)$, there is $\tilde u \in H^1_0(\Omega_a)$ satisfying 
$$\|\nabla u- \nabla \tilde u\|_{L^2 (\R^N)} \leq C \|\nabla u\|_{L^2(W)},$$
where $W$ is a tiny strip close the boundary of $\Omega_b$.

\item[\rm{(iii)}] Finally, we provide an estimate on the energy of the eigenfunctions in proximity of the boundary. In the Dirichlet case, we obtain a decay result by relying on the monotonicity argument by Alt, Caffarelli and Friedman \cite{acf} (see Lemma \ref{monot}). For the Neumann case, we rely on a compactness argument coming from \cite{l1}.
\end{enumerate}

The paper is organized as follows. In Section 2 we go over some properties of Reifenberg-flat domains that we need in the following. Section~\ref{s:estpro} is devoted to the statement and proof of the above-mentioned abstract result (Lemma~\ref{l:ab}). In Section~\ref{s:proofmain} we conclude the proof of Theorem~\ref{main} (Dirichlet problem in Reifenberg-flat domains) and Theorem \ref{mainLip} (Dirichlet problem in Lipschitz domains).  Finally in Section~\ref{s:proofmain2} we establish the proof of Theorem~\ref{main2} (Neumann problem in Reifenberg-flat domains).

\subsection{Notation}
We denote by $C(a_1, \dots, a_h)$ a constant only depending on the variables $a_1, \dots, a_h$. Its precise value can vary from line to line. Also, we use the following notations:  \\
$\Delta$ : the Laplace operator.\\
$\Hh^N$ : the Hausdorff measure of dimension $N$.\\
$|A|$ : the Lebesgue measure of the Borel set $A$. \\
$C^\infty_c(\Omega)$ : $C^\infty$ functions with compact support in the open set $\Omega$.\\
$H^1(\Omega)$ : the Sobolev space of the $L^2(\Omega)$ functions whose distributional gradient is in $L^2(\Omega)$, endowed with the norm 
$$
   \| u \|_{H^1(\Omega)} : =  \sqrt{ \| u \|^2_{L^2(\Omega)} +  \| \nabla u \|^2_{L^2(\Omega)}}. 
$$
$H^1_0(\Omega)$ : the closure of $C^{\infty}_c (\Omega)$ in $H^1 (\Omega)$, endowed with the norm $\| u \|_{H^1_0(\Omega)}: =\| \nabla u \|_{L^2(\Omega)}$.\\
$[x, y]$: the segment joining the points $x, y \in \R^N$. \\
$x \cdot y$: the Euclidian scalar product between the vectors  $x, y \in\R^N$. \\
$|x |$: the Euclidian norm of the vector $x \in \R^N$. \\ 
$d(x, y)$: the distance from the point $x$ to the point $y$, $d(x, y)= |x-y|$. \\
$d(x, A)$: the distance from the point $x$ to the set $A$. \\
$d_H(A,B)$ : the Hausdorff distance from the set $A$ to the set $B$.\\
$\diam (A)$: the diameter of the bounded set $A$,
$$
    \diam (A): = \sup_{x, y \in A} d(x, y)
$$
${\bf 1}_A:$ the characteristic function of the set $A$. \\
$B(x, r)$: the open ball of radius $r$ centered at $x$. \\ 
When no misunderstanding can arise, we write $\| \cdot \|_2$ and $\| \cdot \|_{\infty}$ instead of $\| \cdot \|_{L^2(\Omega)}$ and $\| \cdot \|_{L^{\infty} (\Omega)}$.  \\
\section{Preliminary results}\label{Sprel}
\label{topos}
We start by quoting a covering lemma that we need in the following.
\begin{lem}{\rm \cite[Lemma 11]{lm2}} 
\label{cov1} 
Let $\Omega\subseteq \R^N$ be an $(\varepsilon,r_0)$-Reifenberg-flat domain such that $0<\Hh^{N-1}(\partial \Omega)<+\infty$. Given $r<r_0/2$, consider the family of balls $\{B(x,r)\}_{x \in \partial \Omega}$. We can extract a subfamily $\{ B(x_i, r) \}_{i \in I}$ satisfying the following properties: \begin{itemize}
\item[(i)]$\{ B(x_i, r) \}_{i \in I}$ is a covering of $\displaystyle{ \bigcup_{x \in \partial \Omega}B(x,\frac{4}{5}r)}$; 
\item[(ii)] we have the following bound: $\sharp I \leq C(N) \Hh^{N-1}(\partial \Omega)/ r^{N-1}$; 
\item[(iii)] the covering is bounded, namely  
\begin{equation}
\label{e:disjoint}
        B(x_i, r/10) \cap  B(x_j, r/10) = \emptyset \quad \text{if} \; i \neq j. 
\end{equation}
\end{itemize}
\end{lem}

\begin{remark} Note that all the balls $B(x_1, r) \dots, B(x_{\sharp I}, r)$ are contained in a sufficiently large ball of radius $2\diam(\Omega)$. By combining this observation with property~\eqref{e:disjoint} we get the estimate $\sharp I \leq \diam(\Omega)^{ N} /r^N$. The reason why property (ii) is not trivial is because we provide a bound in terms of $1/r^{N-1}$ instead of $1/r^N$: we need this sharper estimate to achieve for instance~\eqref{e:ref}.

Note also that, by applying a similar argument, we get the estimate 
\begin{equation}
\label{e:overlap}
    \text{for any} \; x \in \bigcup_{i \in I} B(x_i, r), \qquad
    \sharp\big\{i :  x \in B(x_i, 2r) \big\} \leq C(N). 
\end{equation}

\end{remark}

In the following we need cut-off functions $\theta_0, \dots, \theta_{\sharp I}$ satisfying suitable conditions. The construction of these functions is standard, but for completeness we provide it.  
\begin{lem} \label{thetazero}  Under the same hypothesis as in Lemma \ref{cov1}, there  are Lipschitz continuous cut-off function $\theta_i: \mathbb R^N \to \R$, $i=0, \dots, \sharp I$, that satisfy the following conditions:  
\begin{equation}
\label{e:theta} 
\begin{split}
&        0 \leq \theta_i (x) \leq 1, \; \forall \, x \in \R^N,  
        \qquad |\nabla \theta_i (x) | \leq \frac{C(N)}{r} \; a.e. \, x \in \R^N, \qquad i=0, \dots, 
        \sharp I \\
&        \theta_0  (x) =0  \; \; \mathrm{if} \;
        x \in \bigcup_{i \in I} B \left( x_i, r \right), \qquad
        \theta_0  (x) =1 \; \; \mathrm{if} \; x \in  \R^N \setminus \bigcup_{i \in I} B \left( x_i,   2 r \right) \phantom{\int} \\
&      \theta_i  (x) =0 \; \; \mathrm{if} \; x \in  \R^N \setminus B(x_i, 2 r ), 
      \; i=1, \dots, \sharp I, \qquad \sum_{i=0}^{\sharp I} \theta_i (x) =1 , \; \forall \, x \in \R^N.    \\
  \end{split}
       \end{equation}
\end{lem}

\begin{proof} Let $\ell, h: [0, + \infty ) \to [0, 1]$ be the Lipschitz continuous functions defined as follows: 
$$
    \ell(t) : = 
    \left\{
    \begin{array}{lll}
    0 & \text{if} \; 0 \leq t \leq 1 \\
    2 (t- 1) & \text{if} \; 1 \leq t \leq 3/2 \\    
    1 & \text{if} \; t \ge 3/2 \\
    \end{array}    
    \right. \qquad 
     h(t) : = 
    \left\{
    \begin{array}{lll}
    1 & \text{if} \; 0 \leq t \leq 3/2 \\
    -2(t -2 ) & \text{if} \; 3/2  \leq t \leq 2 \\    
    0 & \text{if} \; t \ge 2. \\
    \end{array}    
    \right.
$$
First, we point out that $|\ell '|, |h '| \leq 2$. Next, we set
$$
    \psi_0 (x) : = \prod_{i=1}^{\sharp I} \ell \big(d(x, x_i)/r \big) \qquad \psi_i (x):= h \big(d(x, x_i)/ r \big), \; i=1, \dots, \sharp I.   
$$
The goal is now showing that 
\begin{equation}
\label{e:bound}
         1 \leq \sum_{i=0}^{\sharp I} \psi_i(x) \leq C(N), \quad \forall \, x \in \R^N. 
\end{equation}
To establish the bound from below, we make the following observations: first, all the functions $\psi_i$, $i=0, \dots, \sharp I$ take by construction only nonnegative values. Second, 
$$  
    \psi_0 (x) \equiv 1 \quad \forall \; x \in \R^N \setminus \bigcup_{i=1}^{\sharp I} B(x_i, 3 r /2) 
$$    
and hence the lower bound holds on that set. If $x \in \displaystyle{\bigcup_{i=1}^{\sharp I} B(x_i, 3 r /2) }$, then at least one of the $\psi_i$, $i=1, \dots, \sharp I$, takes the value $1$. This establishes bound from below in~\eqref{e:bound}.

To establish the bound from above, we point out that, for any $i=1, \dots, \sharp I$, 
$\psi_i(x) =0$ if $x \in \R^N \setminus B(x_i, 2r)$. By recalling~\eqref{e:overlap} and that $\psi_i(x) \leq 1$ for every $x \in \R^N$ and $i =0, \dots, \sharp I$, we deduce that
$$
     \sum_{i=0}^{\sharp I} \psi_i(x) \leq C(N) + 1 \leq C(N), \qquad \forall \, x \in \R^N
$$
and this concludes the proof of~\eqref{e:bound}. 

Due to the lower bound in~\eqref{e:bound}, we can introduce the following definitions: 
$$
    \theta_i(x) : = \frac{\psi_i(x)}{ \displaystyle{\sum_{i=0}^{\sharp I} \psi_i(x)}}, \quad i=0, \dots, \sharp I, \; x \in \R^N. 
$$
We now show these functions satisfy~\eqref{e:theta}: the only nontrivial point is establishing the bound on the gradient. To this end, we first point out that 
$$
    |\nabla \psi_i (x) | \leq \frac{2}{r}, \quad \forall \; x \in \R^N, \; i=0, \dots, \sharp I.
$$
Next, we recall that $\psi_i(x) =0$ for every $x \in \R^N \setminus B(x_i, 2r)$ and every $i=1, \dots, \sharp I$. By combining these observations with inequalities~\eqref{e:overlap} and~\eqref{e:bound} we get
\begin{equation*}
\begin{split}
          |\nabla \theta_i (x) | = 
          & 
          \left| 
          \frac{\nabla \psi_i(x)}{ \displaystyle{\sum_{i=0}^{\sharp I} \psi_i(x)}} -
          \frac{\psi_i(x)  \displaystyle{\sum_{i=0}^{\sharp I} \nabla \psi_i(x)}}{\Big( \displaystyle{\sum_{i=0}^{\sharp I} \psi_i(x)} \Big)^2}
          \right| \leq \frac{2}{r} + \frac{C(N)}{r} \leq \frac{C(N)}{r} \quad  \forall \; x \in \R^N, \; i=0, \dots, \sharp I. \\
\end{split}
\end{equation*}
This concludes the proof of the lemma. 
\end{proof}

We now state a result ensuring that the classical Rellich-Kondrachov Theorem applies to Reifenberg-flat domains. The proof is provided in~\cite{GeomPart}. For simplicity, here we only give the statement in the case when the summability index is $p=2$, but the result hold in the general case, see~\cite{GeomPart}. 
\begin{prop}\label{embedding} Let $\Omega \subseteq \mathbb R^N$ be a bounded, connected, $(\varepsilon,r_0)$-Reifenberg-flat domain and assume that $\varepsilon \leq 1/600$. Then the following properties hold:  
\begin{enumerate}
\item[\rm{(i)}] if $N >2$, $H^1(\Omega)$ is continuously embedded in $L^{2^*}(\Omega)$, $2^*:=2N/(N-2)$, and it is compactly embedded in $L^q(\Omega)$ for any $q \in [1, 2^*[$. Also, the norm 
of the embedding operator only depends on $N$, $r_0$, $q$ and on the diameter $\diam (\Omega)$. 
\item[\rm{(ii)}] if $N=2$, $H^1(\Omega)$ is compactly embedded in $L^{q}(\Omega)$ for every $q \in [1, + \infty[$. Also, the norm of the embedding operator only depends on $r_0$, $q$ and $\diam (\Omega)$. 
\end{enumerate}
\end{prop}    
A consequence of Proposition \ref{embedding} is that Neumann eigenfunctions defined in Reifenberg-flat domains are bounded (see again~\cite{GeomPart} for the proof).
\begin{prop}\label{uestim} Let $\Omega \subseteq \mathbb R^N$ be a bounded, connected, $(\varepsilon,r_0)$-Reifenberg-flat domain and let $u$ be a Neumann eigenfunction associated with the eigenvalue $\mu$. If $\varepsilon \leq 1/600$, then $u$ is bounded and 
\begin{eqnarray}
\label{e:linfty}
\|u \|_{L^\infty(\Omega)}\leq C (1+\sqrt{\mu})^{\gamma(N)} \|u\|_{L^{2}(\Omega)}, \label{desirein}
\end{eqnarray}
where $\gamma(N)=\displaystyle{ \max\left\{ \frac{N}{2},\frac{2}{N-1} \right\}}$ and $C=C(N, r_0, \diam (\Omega))$.
\end{prop}
The following result ensures that Dirichlet eigenfunctions satisfy an inequality similar to \eqref{desirein}. Note that in this case no regularity requirement is imposed on the domain. 
\begin{prop}{\cite[Lemma 3.1]{Davies}}  \label{propinfty1} Let $\Omega \subseteq \R^N$ be a bounded domain and let $v$ be an eigenfunction for the Dirichlet Laplacian in $\Omega$ and let $\lambda$ be the associated eigenvalue. Then $v$ is bounded and 
\begin{equation}
\|v\|_{L^\infty(\Omega)}\leq \left(\frac{\lambda e }{2\pi N}\right)^{\frac{N}{4}}\|v\|_{L^2(\Omega)}.\label{infinitybound}
\end{equation}
\end{prop}
To conclude this section we quote a result from~\cite{l1} concerning harmonic functions satisfying mixed Neumann-Dirichlet conditions.  More precisely, let  $\Omega$ be an $(\varepsilon,r_0)$-Reifenberg-flat domain and let $u \in H^1(\Omega)$. Given $x \in \partial \Omega $ and $r<r_0$, consider the problem
\begin{equation}\label{probbb}
\begin{cases}
\Delta v = 0 &\text{ in } \Omega \cap B(x,r)\\
v= u & \text { on } \partial B(x,r) \cap   \Omega \phantom{\displaystyle{\Bigg(}}\\
\displaystyle{\frac{\partial{v} }{  \partial \nu }=0}  &  \text{ on } \partial \Omega \cap B(x,r). 
\end{cases}
\end{equation}
The existence of a weak solution of~\eqref{probbb} can be obtained by considering the variational formulation
\begin{eqnarray}
    \min \Big\{\int_{B(x,r)\cap \Omega}|\nabla w|^2 dx: \; {\displaystyle{ w \in  H^1(B(x,r)\cap \Omega)}, \;  w = u \; \textrm{on}  \; \partial B(x,r) \cap     \Omega}   \Big\} \label{minimi}
\end{eqnarray}
and by then applying the argument in~\cite[Proposition 3.3]{d}. We now state the result that  we need in the following.
\begin{theo}{\rm \cite[Theorem 1.2]{l1}} 
\label{thdecay} 
For every $\beta>0$ and  $a\in \, ]0,1/2[$ there is a positive $\varepsilon= \varepsilon (\beta, a)$ such that the following holds.  Let  $\Omega$ be an $(\varepsilon,r_0)$-Reifenberg-flat domain, $u \in H^1(\Omega)$,  $x \in \partial \Omega $ and $r\leq r_0$. Then the solution of the minimum problem~\eqref{minimi} satisfies 
\begin{eqnarray}
\int_{B(x,ar)\cap \Omega}|\nabla v |^2 dx \leq  a^{N-\beta}
\int_{B(x,r) \cap \Omega}|\nabla v|^2dx .\label{decay3}
\end{eqnarray}
\end{theo}
\begin{remark}Note that the decay estimate in Theorem \ref{thdecay} does not depend on the trace of $u$ on the spherical part $\partial B(x,r) \cap \Omega$. 
\end{remark}
\section{Reduction to projection estimates}
\label{s:estpro}
In this section we study the abstract eigenvalue problem and we establish Lemma~\ref{l:ab}, which roughly speaking says that controlling the behavior of suitable projections is sufficient to control the difference between eigenvalues. The abstract framework is sufficiently general to apply to both the Dirichlet and the Neumann problem: these applications are discussed in Sections~\ref{s:proofmain} and~\ref{s:proofmain2} respectively. 

\subsection{Abstract framework}
Our abstract framework is composed of the following objects:
\begin{enumerate} 
\item[\rm{(i)}] $H$ is a real, separable Hilbert space with respect to the scalar product $\h  (\cdot, \cdot)$, which induces the norm $\| \cdot \|_{\h}$. 
\item[\rm{(ii)}] $h : H \times H \to \R$ is a function satisfying:
\begin{equation}
\label{e:c1}
         h (\alpha u + \beta v, z) = \alpha h (u, z) + \beta h (v, z), \qquad
         h (u, v) = h(v, u), \qquad 
         h (u, u) \ge 0, 
\end{equation}
for all $u$, $v$, $z \in H$ and $\alpha$, $\beta \in \R$. Note that we are not assuming that $h$ is a scalar product, namely $h(u, u)=0$ does not necessarily imply $u =0$. We also assume that  
there is a constant $C_H>0$ satisfying  
\begin{equation}
\label{e:poinab}
    h (u, u ) \leq C_H \h (u, u) \quad \forall \, u \in H. 
\end{equation}
\item[\rm{(iii)}] $V$ is a closed subspace of the Hilbert space $H$ such that 
the restriction of the bilinear form $h (\cdot, \cdot)$ to $V \times V$ is actually a scalar product on $V$, namely 
\begin{equation}\label{assumption3}
{\rm{for \ every}}\ v \in V,\ h(v, v)=0\ {\rm{implies}}\ v=0. 
\end{equation}
We denote by $\bar V$ the closure of $V$ with respect to the norm $\| \cdot  \|_h$ induced by $h$ on $V$ and we also assume that the inclusion 
\begin{equation}
\label{e:rellichab}
        \begin{split}
        i :  \Big( V, \| \cdot  \|_{\h} & \Big) \to  \Big( \bar V, \| \cdot  \|_{h} \Big) \\
        & u \mapsto u \\
       \end{split}        
\end{equation}
is compact. 
\end{enumerate}
We now introduce an abstract eigenvalue problem associated with the function $h$ defined above: we are interested in eigencouples $(u, \lambda) \in V \setminus \{ 0 \} \times \R$ satisfying 
\begin{equation}\label{e:eigpb}
\h (u, v ) = \lambda \, h (u, v)
\end{equation}
for every $v\in V$.

We first check that~\eqref{e:eigpb} admits a solution. In view of ~\eqref{e:c1} 
 and~\eqref{e:poinab} we obtain that, for any fixed $f \in \bar V$, the map
$$
        v \mapsto h (f, v) 
$$
is a linear, continuous operator defined on the Hilbert space $(V, \| \cdot \|_{\h})$. As a consequence, Riesz's Theorem ensures that there is a unique element $u_f \in V$ such that
\begin{equation}
\label{e:t}
    \h (u_f, v ) = h(f, v) \quad \forall \; v \in V. 
\end{equation}
Consider the map 
\begin{equation*}
        \begin{split}
        T :    \Big(\bar V, \| \cdot \|_h & \Big) \to \Big(V, \| \cdot \|_\h \Big) \\
        &  f \mapsto u_f \\
       \end{split}        
\end{equation*}
which is linear and continuous since by setting $v= u_f$ in~\eqref{e:t} and by using~\eqref{e:poinab} one gets $\|u_f \|_\h \leq \sqrt{C_H}    \| f \|_h$. Let $i$ be the same map as in~\eqref{e:rellichab}, then the composition $T \circ i$ is compact. Also, in view of~\eqref{e:t} we infer that it is self-adjoint. By relying on classical results on compact, self-adjoint operators we infer that there is a sequence of eigencouples $(\nu_n, u_n)$ such that 
$$
    u_n \neq 0, \quad T u_n = \nu_n u_n, \quad  \lim_{n \to + \infty} \nu_n = 0, \quad \h (u_n, u_m) =0
     \; \text{if $m \neq n$}.
$$
As usual, each eigenvalue is counted according to its multiplicity. Note that $\nu_n >0$ for every $n$ by (\ref{assumption3}) and~\eqref{e:t}, and that $(u_n, \lambda_n = 1/ \nu_n)$ is a sequence of eigencouples for the eigenvalue problem~\eqref{e:eigpb}.

As a final remark, we recall that the value of $\lambda_n$ is provided by the so-called Reyleigh min-max principle, namely 
\begin{equation}
   \label{e:minmax}
              \lambda_n = 
              \min_{ S \in \mathcal{S}_n  } \max_{ u \in S \setminus \{ 0 \} } \frac{\h (u, u)}{h (u, u)}
              = 
              \max_{ u \in S_n \setminus \{ 0 \}  } \frac{\h (u, u)}{h (u, u)}
             =
             \frac{\h (u_n, u_n)}{h (u_n, u_n)}
              \quad \forall n \in \mathbb{N}.
\end{equation}
In the previous expression, $\mathcal S_n$ denotes the set of subspaces of $V$ having dimension equal to $n$ and $S_n$ denotes the subspace generated by the first $n$ eigenfunctions.

\subsection{An abstract lemma concerning eigenvalue stability}
\label{ss:lemma}
We now consider two closed subspaces $V_a$, $V_b \subseteq H$ satisfying assumptions~\eqref{assumption3} and~\eqref{e:rellichab} above and we denote by $(u^a_n, \lambda^a_n)$ and 
$(u^b_n, \lambda^b_n)$ the sequences solving the corresponding eigenvalue problem~\eqref{e:eigpb} in $V=V_a$ and $V= V_b$ respectively. We fix $n \in \mathbb N$ and as before we denote by $S^b_n$ the subspace 
\begin{equation}
\label{e:san}
     S^b_n = \mathrm{span} \langle u^b_1, \dots, u^b_n \rangle. 
\end{equation}
We denote by $P: H \to V_a$ the projection of $H$ onto $V_a$, namely
\begin{equation}
\label{e:proj}
        \forall \, u \in H, \; \; \h (u- Pu, v ) =0 \; \; \forall \, v \in V_a. 
\end{equation}
By relying on an argument due to Birkhoff, De Boor, Swartz and Wendroff~\cite{BBSW} we get the main result of the present section. 
\begin{lem}
\label{l:ab}
Fix $n \in \mathbb N$ and assume there are constants $A$ and $B$, possibly depending on $n$, such that $A>0$, $0 < B <1$ and, for every $u \in S^b_n$, 
\begin{equation}
\label{e:A}
             \| Pu - u \|^2_{\h} \leq A \| u \|^2_h
\end{equation}
and 
\begin{equation}
\label{e:B}
             \| Pu - u \|^2_{h} \leq B \| u \|^2_h.
\end{equation}
Then 
\begin{equation}
\label{e:TOT}
           \lambda^a_n \leq \lambda^b_n + \frac{A}{(1 - \sqrt B)^2}.
\end{equation}
\end{lem}
\begin{proof}
We proceed in several steps.

\noindent {\sc $\diamond$  Step 1.}  From~\eqref{e:B} we get 
 that for every $u \in S^b_n$ one has
 \begin{equation}
 \label{e:pu}
      \|  P u \|_h = \| P u - u + u \|_h \geq  \| u \|_h - \| P u - u \|_h \geq \| u \|_h \big( 1 - \sqrt B \big).
 \end{equation}
 In particular the restriction of the projection $P$ to $S^b_n$ is injective because $ S^b_n \subseteq V_b$ and hence $\| u\|_{h} =0$ implies $u = 0$. Hence, from the min-max principle~\eqref{e:minmax} we get 
 \begin{equation}
 \label{e:pr2}
            \lambda^a_n \leq  \max_{ u \in S^b_n \setminus \{ 0 \} } \frac{\h (Pu,P u)}{h (Pu, Pu)}. 
 \end{equation}
 \noindent {\sc $\diamond$  Step 2.}  To provide an estimate on the right hand side of~\eqref{e:pr2}, we introduce the auxiliary function $p_n : H \to S^b_n$ defined as follows: for every $z \in H$, $p_n z \in S^b_n$ is the unique solution of the minimum problem
 \begin{equation}
 \label{e:min}
      h ( z- p_n z, z- p_n z) = \inf_{ v \in S^b_n}  \Big\{ h ( z- v, z- v) \Big\}. 
 \end{equation}
 Loosely speaking, $p_n$ is the orthogonal projection of  $H$ onto $S^b_n$ with respect to the bilinear form $h$. The argument to show that the minimum problem~\eqref{e:min} admits a solution which is also unique is standard, but for completeness we provide it. 
 
We introduce a minimizing sequence $\{ v_k \} \subseteq S^b_n$, then we have that the sequence ${h ( z- v_k, z- v_k)}$ is bounded. From (\ref{e:c1}) we infer that the bilinear form $h$ satisfies the Cauchy-Schwarz inequality and hence that 
$$
    \sqrt{ h ( z- v_k, z- v_k) } \ge \sqrt{ h (v_k, v_k)} - \sqrt{ h (z, z)}, 
$$  
which implies that the sequence  $ h (v_k, v_k)$ is also bounded. Since $S^b_n \subseteq V_b$, then by (\ref{assumption3}) we have that the bilinear form $h$ is actually a scalar product on $S^b_n$. Since $S^b_n$ has finite dimension, from the bounded sequence $\{ v_k \}$ we can extract a converging subsequence  $\{ v_{k_j} \}$, namely  $h (v_{k_j} - v_0, v_{k_j} - v_0) \to 0$ as $j \to + \infty$, 
 for some $v_0 \in S^b_n$. Hence, 
  $$
       h ( z-  v_{k_j}, z- v_{k_j}) \to 
         h ( z- v, z- v_0) 
 $$
as $j \to + \infty$. This implies that $v_0$ is a solution of~\eqref{e:min} and we set $v_0= p_n z$. To establish uniqueness, we first observe that, given $z\in H$, any $p_n z$ solution of the minimization problem~\eqref{e:min} satisfies the Euler-Lagrange equation 
\begin{equation}
\label{e:el}
          h (z - p_n z, v) =0 \quad \forall \; v \in S^b_n. 
\end{equation}
Assume by contradiction that~\eqref{e:min} admits two distinct solutions $p^1_n z$ and $p^2_n z$, then from~\eqref{e:el} we deduce that $h (p^1_n z - p^2_n z, v) =0$ for every $v \in S^b_n$. By taking $v= p^1_n z - p^2_n z$ and in view of (\ref{assumption3}) we obtain $p^1_n z = p^2_n z$. \\
\noindent {$\diamond$  \sc Step 3.} Next we show that
 \begin{equation}
 \label{e:keye}
          \lambda_n^a \leq \lambda_n^b + \max_{ u \in S^b_n \setminus \{ 0 \} } \frac{\h (Pu - p_n \circ Pu,Pu - p_n \circ Pu)}{h (Pu, Pu)},
 \end{equation}
where $ p_n \circ P$ denotes the composition of $p_n$ and $P$. To obtain~\eqref{e:keye} we observe that for every $u \in S^b_n$ we have 
 \begin{equation}
 \label{e:st2}
 \begin{split}
           \h (Pu,P u)  & =  \h (Pu - p_n \circ P u + p_n \circ P u , Pu - p_n \circ P u + p_n \circ P u )  \\
         & =  \h   (Pu - p_n \circ P u , Pu - p_n \circ P u  ) + \h  (p_n \circ P u , p_n \circ P u ),   \\
 \end{split}
 \end{equation}
where we have used that
\begin{equation}
\label{e:perp}
          \h (Pu - p_n \circ P u  , p_n \circ P u ) =0. 
\end{equation}
Note also that 
\begin{equation}
\label{e:ine}
           h ( Pu, Pu) \geq h (p_n \circ P u, p_n \circ P u)  
\end{equation}
for every $u \in S^b_n$. We assume for the moment that (\ref{e:perp}) and (\ref{e:ine}) are true as their proof will be provided in Step 5 and Step 6.

Combining~\eqref{e:st2} and~\eqref{e:ine} we obtain 
that for every $u \in S^b_n \setminus \{ 0\}$ such that $p_n \circ P u \neq 0$ we have 
\begin{equation}
\label{e:b1}
\begin{split}
         \frac{\h (Pu,P u)}{h (Pu, Pu)} & =
         \frac{ \h   (Pu - p_n \circ P u , Pu - p_n \circ P u  )}{h (Pu, Pu)} + 
         \frac{\h  (p_n \circ P u , p_n \circ P u )}{{h (Pu, Pu)}} \\ & \leq 
         \frac{ \h   (Pu - p_n \circ P u , Pu - p_n \circ P u  )}{h (Pu, Pu)} +  
         \frac{\h  (p_n \circ P u , p_n \circ P u )}{{h (p_n \circ Pu, p_n \circ Pu)}} \\ & \leq
         \frac{ \h   (Pu - p_n \circ P u , Pu - p_n \circ P u  )}{h (Pu, Pu)}  +
         \lambda_n^b.              \\
\end{split}         
\end{equation}
In the previous chain of inequalities we have used that $p_n \circ P$ attains values in $S_n^b$ and that 
$$
    \lambda_n^b = \max_{ w \in S_n^b \setminus \{ 0\}} \frac{\h (w, w)}{ h(w, w)}.
$$
If $p_n \circ P u = 0$, we have 
\begin{equation}
\label{e:b2}
    \frac{\h (Pu,P u)}{h (Pu, Pu)}  \leq \frac{ \h   (Pu - p_n \circ P u , Pu - p_n \circ P u  )}{h (Pu, Pu)}  +
         \lambda_n^b
 \end{equation}        
since $\lambda_n^b > 0$. By combining~\eqref{e:b1} and~\eqref{e:b2} with~\eqref{e:pr2} we eventually get~\eqref{e:keye}. \\
\noindent {\sc $\diamond$  Step 4.}
We now conclude the argument by establishing~\eqref{e:TOT}. First, observe that, for every $u \in S^b_n$, 
one has  
\vspace{0.2cm}

$\displaystyle{ \h   (Pu - p_n \circ P u , Pu - p_n  \circ P u  ) }$
\begin{eqnarray}
\label{e:b3}
       & \leq& \h (Pu - p_n \circ P u , Pu - p_n \circ P u  )   +  \h ( p_n \circ P u - u,  p_n \circ P u - u )   \notag \\
         &=& \h   (Pu - p_n \circ Pu + p_n \circ Pu  - u,  Pu - p_n \circ Pu + p_n \circ Pu  - u  ) \notag \\
         &=& \h   (Pu - u   , Pu - u  ). \notag 
\end{eqnarray}
In the previous formula we have used the equality
\begin{equation}
\label{e:st3}  
           \h  (Pu - p_n \circ P u  ,  p_n \circ P u  - u) =0 ,  
\end{equation}
which is proved in Step 5. 
Combining~\eqref{e:b2},~\eqref{e:b3} along with~\eqref{e:A} and~\eqref{e:pu} we finally obtain
\begin{equation}
\label{e:final}
       \lambda_n^a \leq \lambda_n^b  + \max_{u \in S^b_n \setminus \{0 \}}\frac{\h  (Pu - u   , Pu - u  )}{ h ( Pu, Pu)}
       \leq  \lambda_n^b +  \frac{A \, h (u, u)}{ (1 - \sqrt{B}) h (u, u) },
\end{equation}
which concludes the proof of Lemma~\ref{l:ab} provided that~\eqref{e:perp},~\eqref{e:ine} and~\eqref{e:st3} are true.

{\sc $\diamond$ Step 5. } We establish the proof of~\eqref{e:perp} and~\eqref{e:st3}. Observe that, if $z \in V_b$, then the following implication holds true:
\begin{equation}
\label{e:impli}
    h (z, v) = 0 \; \; \forall \, v \in S^b_n \implies \h (z, v) = 0 \; \; \forall \, v \in S^b_n. 
\end{equation}
Indeed, $v$ is a linear combination of the eigenfunctions $u^b_1, \dots, u^b_n$ and hence 
$$
      \h (z, v) =   \sum_{i=1}^n v_i \h (z, u^b_i) =  \sum_{i=1}^n \lambda^b_i v_i h (z, u^b_i) =0. 
$$
To establish~\eqref{e:perp}, we observe that $p_n \circ Pu \in S^b_n$. Also, by the property~\eqref{e:el} of the projection $p_n$ we have $h( Pu- p_n \circ Pu, v) =0 $ for every $ v \in S^b_n$. Hence by applying~\eqref{e:impli} we obtain~\eqref{e:perp}. To prove~\eqref{e:st3}
we can repeat the same argument. Indeed, if $u  \in S^b_n$, then $[u - p_n \circ Pu ]  \in S^b_n$ and~\eqref{e:st3} follows from~\eqref{e:impli}.

{\sc $\diamond$ Step 6. } Finally we establish~\eqref{e:ine}. Note that~\eqref{e:el} implies 
$$
    h (p_n \circ P u, Pu-p_n \circ Pu) =0
$$
for every $u \in S^b_n$. Hence 
\begin{equation*}
\begin{split}
          h (p_n \circ Pu, p_n \circ Pu) & \leq h (p_n \circ Pu, p_n \circ Pu) + 
           h (Pu - p_n \circ Pu, Pu - p_n \circ Pu) \\
          & = h (Pu, Pu)
\end{split}
\end{equation*}
for every $u \in S^b_n$.
\end{proof}

\section{Stability estimates for Dirichlet eigenvalues}
\label{s:proofmain}
\subsection{The Dirichlet eigenvalue problem in Reifenberg-flat domains}
Given an open and bounded set $\Omega$, $(u, \lambda) \in (H^1_0 (\Omega) \setminus \{ 0\} )\times \mathbb R$ is an eigencouple for the Dirichlet Laplacian in $\Omega$ if
\begin{equation}
\label{e:dpw}
          \int_{\Omega} \nabla u (x) \cdot \nabla v (x) dx = \lambda \int_{\Omega}
         u (x) \, v (x) dx \qquad \forall \,  v \in H^1_0 (\Omega).
\end{equation}
In this section, we apply the abstract framework introduced in Section~\ref{s:estpro} to study how the eigenvalues $\lambda$ satisfying~\eqref{e:dpw} depend on $\Omega$. Let $\Omega_a$ and $\Omega_b$ two Reifenberg-flat domains and denote by $D$ a ball containing both $\Omega_a$ and $\Omega_b.$ We set 
$$
   H : = H^1_0 (D), \quad \mathcal H(u, v) : = \int_D \nabla u (x) \cdot \nabla v (x) dx, \quad h(u, v) : = \int_D u(x) v(x) dx.  
$$
Note that~\eqref{e:poinab} is satisfied because of Poincar\'e-Sobolev inequality. We also set
$$
    V_a : = H^1_0 (\Omega_a) \qquad V_b : = H^1_0 (\Omega_b)  
$$
and observe that they can be viewed as two subspaces of $H^1_0 (D)$ by extending the functions of $H^1_0 (\Omega_a)$ and $H^1_0 (\Omega_b) $ by $0$ outside $\Omega_a$ and $\Omega_b$ respectively. The compactness of the inclusion~\eqref{e:rellichab} is guaranteed by Rellich-Kondrachov Theorem. Also, the projection $P=P^{\mathcal D}_{\Omega_a}$ defined by~\eqref{e:proj} satisfies in this case 
\begin{equation}
\label{e:pda}
         ||\nabla (u - P^{\mathcal D}_{\Omega_a}u)||_{L^2(D)} = \min_{v \in H^1_0 (\Omega_a)} \big\{ ||\nabla (u - v)||_{L^2(D)}  \big\}.
\end{equation}
By using the above notation, the Dirichlet problem~\eqref{e:dpw} reduces to~\eqref{e:eigpb} and hence it is solved by a sequence of eigencouples $(\lambda_n, u_n)$ with $\lambda_n >0$ for every $n$ and $\lim_{n \to + \infty}\lambda_n = + \infty.$ We employ the same notation as in Section~\ref{ss:lemma} and we denote by $\{ \lambda^a_n\}$ and $\{ \lambda_n^b \}$ the eigenvalues of the Dirichlet problem in $\Omega_a$ and $\Omega_b$, respectively. By applying Lemma~\ref{l:ab}, we infer that to control the difference $|\lambda_n^a - \lambda_n^b|$ it is sufficient to provide an estimate on suitable projections. This is the content of the following subsection. 

\subsection{Estimates on the projection of Dirichlet eigenfunctions}

\subsubsection{Boundary estimate on the gradient of Dirichlet eigenfunctions}

We now establish a decay result for the gradient of Dirichlet eigenfunctions. The statement is already in \cite[Proposition 14]{lm2}, but the proof contains a gap. This is why we hereafter give a different and complete proof,  which is based on techniques from Alt, Caffarelli and Friedman \cite{acf}. We begin with a monotonicity Lemma.

\begin{lem} \label{monot} Let $\Omega \subseteq \R^N$ be a bounded domain and let $x_0 \in \partial \Omega$. Given a radius $r>0$, we denote by $\Omega_r^+:=B(x_0,r)\cap \Omega$, by $S_r^+:=\partial B(x_0,r) \cap \Omega$ and by $\sigma(r)$ the first Dirichlet eigenvalue of the Laplace operator on the spherical domain $S_r^+$. If there are constants $r_0>0$ and $\sigma^{\ast} \in ]0, N-1[$ such that
$$ 
    \inf_{0< r < r_0} (r^{2}\sigma(r)) \ge \sigma^*,
$$ then the following holds. 
If $(u, \lambda)$ is an eigencouple for the Dirichlet Laplacian in $\Omega$, then the function 
\begin{equation}
\label{e:f}r\mapsto \Bigg(\frac{1}{r^\beta}\int_{\Omega_r^+} \frac{|\nabla u |^2}{|x-x_0|^{N-2}} \;dx \Bigg)+C_0 r^{2-\beta}
\end{equation}
is non decreasing on $]0,r_0[$. In the previous expression, the exponent $\beta \in ]0, 2[$ is defined by the formula
$$
     \beta:  =\sqrt{(N-2)^2 +4\sigma^*}-(N-2) 
$$
and $C_0$ is a suitable constant satisfying 
\begin{equation}
\label{constantC}
    C_0 =  \frac{\beta C (N)}{(2-\beta)}\lambda \|u\|_{\infty}^2.
\end{equation}
We also have the bound
\begin{eqnarray}
         \int_{\Omega_{r_0/2}^+} \frac{|\nabla u |^2}{|x-x_0|^{N-2}} \;dx \leq   
          C(N,r_0,\beta)\lambda(1+\lambda^{\frac{N}{2}})\|u\|_2^2,
                       \label{boundM}
\end{eqnarray}
{where 
$$
   C(N,r_0,\beta) =C(N){2^{\beta} r_0^{-\beta}} \left[r_0^{2-N}\frac{1}{\beta}+r_0^{3\beta}+\frac{\beta}{2-\beta}2^{-\beta}r_0^2\right].
$$}
\end{lem}
\begin{proof} 
We assume without loss of generality that $x_0=0$ and to simplify notation we denote by $B_r$ the ball $B(0, r)$. Also, in the following we identify $u \in H^1_0(\Omega)$ with the function $u \in H^1(\R^N)$ obtained by setting 
$u(x) =0$ if $x \in \Omega^c$. 

The proof is based on the by now standard monotonicity Lemma of Alt, Caffarelli and Friedman \cite[Lemma 5.1]{acf} and is divided in the following steps. 

\noindent {\sc $\diamond$ Step 1.} We prove the following inequality: for a.e. $r>0$,
\begin{eqnarray}
2\int_{\Omega_r^+}|\nabla u|^2 |x|^{2-N}dx &\leq & r^{2-N}\int_{S_r^+} 2u \frac{\partial u}{\partial \nu}dS +(N-2)r^{1-N}\int_{S^+_r} u^2dS \notag \\
&\quad & \quad \quad + \;2\lambda\int_{\Omega^+_r}u^2 |x|^{2-N} dx. \label{mono00}
\end{eqnarray}
We recall that the eigenfunction $u \in C^{\infty} (\Omega)$ (see the book by Gilbarg and Trudinger~\cite[page 214]{gt}). Although~\eqref{mono00} can be formally obtained through an integration by parts, the rigorous proof is slightly technical. Given $\varepsilon>0$, we set
$$|x|_{\varepsilon}:= \sqrt{x_1^2+x_2^2+\dots +x_N^2 +\varepsilon},$$
so that $|x|_{\varepsilon}$ is a $C^\infty$ function.  A direct computation shows that
$$\Delta(|x|_{\varepsilon}^{2-N})=(2-N)N\frac{\varepsilon}{|x|_{\varepsilon}^{N+2}}\leq 0,$$ 
in other words $|x|_{\varepsilon}^{2-N}$ is superharmonic.

Let $u_n \in C^\infty_c(\Omega)$ be a sequence of functions converging in $H^1(\R^N)$ to $u$. By using the equality
\begin{eqnarray}
\Delta(u_n^2)=2|\nabla u_n|^2 + 2 u_n \Delta u_n  \label{mono1}
\end{eqnarray}
 we deduce that
\begin{eqnarray}
2\int_{\Omega^+_r}|\nabla u_n|^2 |x|_\varepsilon^{2-N} = \int_{\Omega^+_r} \Delta(u_n^2) |x|_\varepsilon^{2-N} - 2\int_{\Omega^+_r}(u_n \Delta u_n) |x|_\varepsilon^{2-N} . \label{mono2}
\end{eqnarray}
Since $\Delta( |x|_\varepsilon^{2-N})\leq 0$,  the Gauss-Green Formula yields
\begin{eqnarray}
\int_{\Omega^+_r} \Delta(u_n^2) |x|_\varepsilon^{2-N}dx  = \int_{\Omega^+_r} u_n^2 \Delta( |x|_\varepsilon^{2-N})dx + I_{n,\varepsilon} (r) \leq I_{n,\varepsilon} (r) \;\label{mono3},
\end{eqnarray}
where 
$$I_{n,\varepsilon}(r)=(r^2+\varepsilon)^{\frac{2-N}{2}}\int_{\partial \Omega^+_r} 2u_n \frac{\partial u_n}{\partial \nu}dS +(N-2)\frac{r}{(r^2+\varepsilon)^{\frac{N}{2}}}\int_{\partial \Omega^+_r} u_n^2dS.$$ 
In other words, \eqref{mono2} reads
\begin{eqnarray}
2\int_{\Omega^+_r}|\nabla u_n|^2 |x|_\varepsilon^{2-N}dx \leq I_{n,\varepsilon} (r)- 2\int_{\Omega^+_r}(u_n \Delta u_n) |x|_\varepsilon^{2-N} dx. \label{mono3bis}
\end{eqnarray}
We now want to pass to the limit, first as  $n\to +\infty$, and then as $\varepsilon\to 0^+$. To tackle some technical problems, we first integrate over $r \in [r,r+\delta]$ and divide by $\delta$, thus obtaining
\begin{eqnarray}
\frac{2}{\delta}\int_{r}^{r+\delta}\left(\int_{\Omega^+_{\rho}}|\nabla u_n|^2 |x|_\varepsilon^{2-N}dx\right) d \rho\leq  A_n-R_n, \label{mono3ter}
\end{eqnarray}
where 
$$A_n=\frac{1}{\delta}\int_{r}^{r+\delta}I_{n,\varepsilon}(\rho) d\rho$$
and 
$$R_n= 2\frac{1}{\delta}\int_{r}^{r+\delta}\left(\int_{\Omega^+_\rho}(u_n \Delta u_n) |x|_{\varepsilon}^{2-N} dx\right) d\rho.$$
First, we investigate the limit of $A_n$ as $n \to + \infty$: by applying the coarea formula, we rewrite $A_n$ as
$$
    A_n = \frac{1}{\delta} \Bigg( 2 \int_{\Omega^+_{r + \delta} \setminus \Omega^+_r}   (|x|^2+\varepsilon)^{\frac{2-N}{2}} u_n \, \nabla u_n \cdot {\frac{x}{|x|}} \, dx  +
    (N-2) \int_{\Omega^+_{r + \delta} \setminus \Omega^+_r}  \frac{|x|}{(|x|^2+\varepsilon)^{\frac{N}{2}}  } u_n^2 dx \Bigg). 
$$
Since $u_n$ converges to $u$ in $H^1(\R^N)$ when $n\to +\infty$, then by using again the coarea formula we get that 
$$A_n \longrightarrow \frac{1}{\delta}\int_{r}^{r+\delta}I_{\varepsilon}(\rho) d\rho \quad  \qquad n \to + \infty, $$
where
$$I_{\varepsilon} (\rho)=(\rho^2+\varepsilon)^{\frac{2-N}{2}}\int_{\partial \Omega^+_{\rho}} 2u \frac{\partial u}{\partial \nu}dS +(N-2)\frac{\rho}{(\rho^2+\varepsilon)^{\frac{N}{2}}}\int_{\partial \Omega^+_{\rho}} u^2dS.$$
Next, we investigate the limit of $R_n$ as $n \to + \infty$. By using Fubini's Theorem, we can rewrite $R_n$ as
$$R_n=\int_{\Omega}(u_n\Delta u_n)f(x)dx,$$ 
where $$f(x)=|x|_\varepsilon^{2-N}\frac{{ 2}}{\delta}\int_{r}^{r+\delta}{\bf 1}_{\Omega^+_\rho}(x)d\rho.$$
Since 
$$
\frac{1}{\delta}\int_{r}^{r+\delta}{\bf 1}_{\Omega^+_\rho}(x)d\rho = 
\left\{
\begin{array}{ll}
1 &\text{ if } x \in \Omega^+_{r}\\
\displaystyle{\frac{r+\delta-|x|}{\delta}} & \text{ if } x \in \Omega^+_{r+\delta}\setminus \Omega^+_{r}\\
0 &\text{ if } x \notin \Omega^+_{r+\delta}\\
\end{array}
\right.,
$$
then $f$ is Lipschitz continuous and hence by recalling $u_n \in C^{\infty}_c (\Omega)$ we get  
\begin{equation*}
\begin{split}
&      \left| \int_{\Omega} \Big( u_n \Delta u_n - u \Delta u \Big) f dx     \right| \leq 
         \left| \int_{\Omega} \Big( \Delta u_n -  \Delta u \Big) u_n f dx     \right| +
          \left| \int_{\Omega} \Big( u_n - u \Big) \Delta u  f dx     \right| \\ 
 &    =  \left| \int_{\Omega} \Big( \nabla u_n -  \nabla u \Big)  \Big( u_n \nabla f + \nabla u_n f \Big) dx     \right|     
      +
      \left| \int_{\Omega}\nabla u \Big( ( \nabla u_n - \nabla u )    f + ( u_n - u ) \nabla f \Big)dx     \right|  \\
 & \leq      \|  \nabla u_n -  \nabla u \|_{L^2 (\Omega)}  
   \Big( \| u_n \|_{L^2 (\Omega)} \| \nabla f \|_{L^{\infty} (\Omega)} + 
  \| \nabla u_n \|_{L^{2} (\Omega)} \|f \|_{L^{\infty} (\Omega) } \Big)  \phantom{\int} \\ 
  & \quad + 
  \| \nabla u  \|_{L^2(\Omega)} \Big(  \|  \nabla u_n -  \nabla u \|_{L^2 (\Omega)}    \|f \|_{L^{\infty} (\Omega)} +
  \|  u_n -   u \|_{L^2 (\Omega)}    \|\nabla f \|_{L^{\infty} (\Omega)}\Big)  \phantom{\int} \\
\end{split}
\end{equation*}
and hence the expression at the first line converges to $0$ as $n \to + \infty$. 

By combining the previous observations and by recalling that $u$ is an eigenfunction we infer that by passing to the limit $n \to \infty$ in~\eqref{mono3ter} we get 
$$
   \frac{2}{\delta}\int_{r}^{r+\delta}\left(\int_{\Omega^+_{\rho}}|\nabla u|^2 |x|_\varepsilon^{2-N}dx\right) dr
   \leq \frac{1}{\delta}\int_{r}^{r+\delta}I_{\varepsilon}(\rho) d\rho - \frac{2}{\delta}\int_{r}^{r+\delta}\left(\int_{\Omega^+_\rho}\lambda u^2 |x|_{\varepsilon}^{2-N} dx\right) d\rho.
$$
Finally, by passing to the limit $\delta \to 0^+$ and then $\varepsilon \to 0^+$ we obtain~\eqref{mono00}. \\
\noindent {\sc $\diamond$ Step 2.} We provide an estimate on the right hand side of~\eqref{mono00}. First, we observe that 
\begin{eqnarray}
\int_{\Omega_r^+}u^2 |x|^{2-N}dx\leq   \|u\|_{\infty}^2 \int_{\Omega^+_r} |x|^{2-N} dx  { \leq } C (N)  \|u\|_{\infty}^2r^2 \label{mono8}
\end{eqnarray}
and we recall that $C(N)$ denotes a constant only depending on $N$, whose exact value can change from line to line. 

Next, we point out that the definition of $\sigma^\ast$ implies that
\begin{eqnarray}
 \int_{S_r^+} u^2dS  \leq  \frac{1}{\sigma^*}r^2 \int_{S_r^+} |\nabla_{\tau}u |^2dS \qquad 
 r\in \, ]0,r_0[, \label{mono4}
 \end{eqnarray}
where $\nabla_{\tau}$ denotes the tangential gradient on the sphere. Also, let $\alpha >0$ be a parameter that will be fixed later, then by combining Cauchy-Schwarz inequality, \eqref{mono4} and the inequality $\displaystyle{ab\leq \frac{\alpha} {2}a^2 + \frac{1}{2 \alpha} b^2}$, we get
\begin{eqnarray} 
\Bigg|\int_{S_r^+} u \frac{\partial u}{\partial \nu}dS \Bigg|&\leq& \Bigg(\int_{S^+_r} u^2 dS\Bigg)^\frac{1}{2} \Bigg(\int_{S^+_r} |\frac{\partial u}{\partial \nu}|^2dS \Bigg)^\frac{1}{2} \notag \\
&\leq &  \frac{r}{\sqrt{\sigma^*}} \Bigg(\int_{S^+_r} |\nabla_\tau u|^2dS \Bigg)^\frac{1}{2} \Bigg(\int_{S^+_r} |\frac{\partial u}{\partial \nu}|^2dS \Bigg)^\frac{1}{2} \notag \\
&\leq &  \frac{r}{\sqrt{\sigma^*}}  \Bigg( \frac{\alpha}{2}\int_{S^+_r} |\nabla_\tau u|^2dS + \frac{1}{2\alpha}\int_{S^+_r} |\frac{\partial u}{\partial \nu}|^2dS \Bigg). \label{mono5}
\end{eqnarray}
Hence, 
\begin{equation}
\label{e:rhs}
\begin{split}
          r^{2-N}\int_{S_r^+} 2u \frac{\partial u}{\partial \nu}dS +(N-2)r^{1-N}\int_{S^+_r} u^2dS  & \leq 
        r^{2-N} \frac{2 r }{\sqrt{\sigma^\ast}} \left[ \frac{\alpha}{2}  
        \int_{S_r^+}  |\nabla_{\tau} u |^2 dS+ 
        \frac{1}{2\alpha}\int_{S^+_r} |\frac{\partial u}{\partial \nu}|^2dS    \right]   \\
        & \quad + (N-2) r^{1-N} \frac{1}{\sigma^\ast} r^2
         \int_{S_r^+}  |\nabla_{\tau} u |^2 dS \\
\end{split}
\end{equation}
 $$\leq r^{3-N} \left[ \left( \frac{\alpha}{\sqrt{\sigma^{\ast}}}  + \frac{N-2}{\sigma^\ast }\right) \int_{S_r^+}  |\nabla_{\tau} u |^2 dS +\frac{1}{\alpha \sqrt{\sigma^\ast}} 
         \int_{S_r^+} |\frac{\partial u}{\partial \nu}|^2 dS   \right]. $$

Next, we choose $\alpha>0$ in such a way that
$$\frac{\alpha}{\sqrt{\sigma^*}} + \frac{N-2}{\sigma^*}=\frac{1}{\alpha\sqrt{\sigma^*}},$$
namely 
$$\alpha = \frac{1}{2\sqrt{\sigma^*}}\big[\sqrt{(N-2)^2 +4\sigma^*}-(N-2) \big].$$
Hence, by combining~\eqref{mono00},~\eqref{mono4} and~\eqref{e:rhs}  
we finally get
\begin{eqnarray}
\int_{\Omega^+_r}|\nabla u|^2 |x|^{2-N}dx \leq  r^{3-N} \gamma(N,\sigma^*) \int_{S_r^+}|\nabla u|^2dS +  C (N) \lambda \|u\|_{\infty}^2 r^{2} \label{mono9},
\end{eqnarray}
where
$$\gamma(N,\sigma^*)=\Big[\sqrt{(N-2)^2 +4\sigma^*}-(N-2)\Big]^{-1}.$$
\noindent {\sc $\diamond$ Step 3.} We establish the monotonicity property. 
We set 
$$
    f(r)=\int_{\Omega_r^+}|\nabla u|^2 |x|^{2-N}dx
$$ 
and we observe that 
$$
     f'(r)=r^{2-N} \int_{S_r^+}|\nabla u|^2
     \quad a.e. \, r\in \, ]0,r_0[\, , 
$$  
hence~\eqref{mono9} implies that
\begin{eqnarray}
f(r) \leq    \gamma r f'(r) +  K r^{2},\label{mono10}
\end{eqnarray}
with $\gamma=\gamma(N,\sigma^*)$ and $K= C(N) \lambda \|u\|_{\infty}^2$.  This implies that 
$$\Big(\frac{f(r)}{r^\beta} + \frac{K \beta}{(2-\beta)}r^{2-\beta}\Big)'\geq 0,$$
with $\beta=1 / \gamma=\sqrt{(N-2)^2 +4\sigma^*}-(N-2) \in ]0,2[$. This establishes the monotonicity result. \\
\noindent {\sc $\diamond$ Step 4.}
We establish~\eqref{boundM}.

First, we observe that by combining the coarea formula with Chebychev inequality we get that for any $a>0$
$$
   \mathcal{H}^1\Big(\{t \in [r_0/2,r_0] : \int_{S_{t}^+} |\nabla u|^2 dS \geq a \}\Big)\leq \frac{1}{a}\int_{\Omega^+_{r_0}\setminus \Omega^+_{r_0/2}}|\nabla u|^2\;dx.$$
By applying this inequality with $a =\displaystyle{ \frac{4}{r_0}\int_{\Omega^+_{r_0}\setminus \Omega_{r_0/2}^+}|\nabla u|^2\;dx}$ we get that there is at least a radius $r_1\in�[r_0,r_0/2]$ such that  
$$\int_{S_{r_1}^+} |\nabla u|^2 dS\leq \frac{4}{r_0} \int_{\Omega_{r_0}^+}|\nabla u|^2\;dx\leq \frac{4}{r_0}\|\nabla u\|^2_{L^2(\Omega)}.$$
By combining~\eqref{mono9} with the fact that $r_0/2\leq r_1\leq r_0$ we infer
$$\int_{\Omega_{r_1}^+}|\nabla u|^2 |x|^{2-N}dx \leq C (N) r_0^{2-N} \gamma \|\nabla u\|^2_{L^2(\Omega)}+  C (N) \lambda \|u\|_{\infty}^2 r_0^{2}$$
and by monotonicity we have that 
\begin{eqnarray}
 (r_0/2)^{-\beta}\int_{\Omega_{r_0/2}^+} \frac{|\nabla u |^2}{|x-x_0|^{N-2}} \;dx  &\leq&  r_1^{-\beta}\int_{\Omega_{r_1}^+} \frac{|\nabla u |^2}{|x-x_0|^{N-2}} \;dx + C_0 {r_1}^{2-\beta} \notag \\
 &\leq & 2^{\beta}C (N)r_0^{2-N-\beta} \gamma \|\nabla u\|_{L^2(\Omega)}^2+2^{\beta}C (N) \lambda \|u\|_{\infty}^2 r_0^{2\beta} +C_0 r_0^{2-\beta} .\notag 
 \end{eqnarray}
Finally, by using that  $\displaystyle{C_0=\frac{\beta C (N)}{(2-\beta)}\lambda \|u\|_{\infty}^2}$, that $\|u\|_{\infty}\leq C(N)\lambda^{\frac{N}{4}} \|u\|_{L^2(\Omega)}$ (see Proposition \ref{propinfty1}), that $\|\nabla u\|_{L^2(\Omega)}^2=\lambda \|u\|_{L^2(\Omega)}^2$, and that $\max \{ \lambda,\lambda^{1+N/2} \}\leq \lambda(1+\lambda^{\frac{N}{2}})$ we get~\eqref{boundM}.
\end{proof}

We now provide an estimate on the energy of a Dirichlet eigenfunction near the boundary.

\begin{prop}
\label{propdecayDir} 
For every $\eta \in \, ]0,1[$ there is a positive constant $\varepsilon:=\varepsilon(\eta)$ such that the following holds. Given $r_0 \in \, ]0, 1[$, let $\Omega \subseteq \R^N$  be  an $(\varepsilon, r_0)$-Reifenberg-flat domain, $x_0 \in \partial \Omega$ and let $u$ be a Dirichlet eigenfunction in $\Omega$ associated with the eigenvalue $\lambda$. Then  
\begin{eqnarray}
\label{e:bdecay1}
           \int_{B(x_0,r)\cap \Omega}|\nabla u|^2\dd x \leq
            C \lambda(1+\lambda^{\frac{N}{2}})r^{N-\eta} \|u\|_2^2 \; \quad \forall \; r \in ]0,r_0/2[ \, , \label{dec00}
\end{eqnarray}
for a suitable positive constant $C= C(N, r_0, \eta)$.
\end{prop}
\begin{proof} We fix $\eta\in \, ]0,1[$ and we recall that the first eigenvalue of the spherical Dirichlet Laplacian on a half sphere is equal to $N-1$. For $t \in \, ]-1,1[$, let $S_t$ be the spherical cap $S_t:= \partial B(0,1) \cap \{x_N >t\}$, so that $t=0$ corresponds to a half sphere. Let $\lambda_1(S_t)$ be the first Dirichlet eigenvalue in $S_t$. In particular, $t\mapsto \lambda_1(S_t)$ is monotone in $t$. Therefore, since $\eta<1$ {and $\lambda_1 (S_t) \to 0$ as $t \downarrow -1$, there  is $t^*(\eta)<0$ such that 
$$
    \lambda_1(t^*)  {\leq}  N-1 - \frac{\eta}{4}(2N -\eta).
$$
{By relying on Lemma 5 in~\cite{GeomPart} (see also Proposition 2.2 in \cite{hm3})}, we infer that, if ${\varepsilon < t^* (\eta) /2}$, then $\partial B(x_0,r) \cap \Omega$ is contained in a spherical cap homothetic to $S_{t^*}$ for every $r \leq r_0$. 
Since the eigenvalues scale of by factor $r^2$ when the domain expands of a factor $1/r$,} by the monotonicity property of the eigenvalues with respect to domains inclusion, we have 
\begin{equation}
\label{e:inf}
\inf_{r<r_0} r^{2}\lambda_1(\partial B(x_0,r) \cap \Omega) \geq \lambda_1(S_{t^*}) { \ge } N-1 - \frac{\eta}{4}(2N -\eta).
\end{equation}
 As a consequence, we can apply Lemma \ref{monot} which ensures that, if $u$ is a Dirichlet eigenfunction in $\Omega$ and $x_0\in \partial \Omega$, then~\eqref{e:f}
is a non decreasing function of $r$, provided that 
\begin{equation}
\label{e:beta}
\beta=\sqrt{(N-2)^2 +4(N-1) - \eta (2N-\eta) } -(N-2)= 2 - \eta
\end{equation}
and $C_0$ is the same as in~\eqref{constantC}.
In particular, by monotonicity we know that for every $r\leq r_0<1$,
\begin{eqnarray}
\frac{1}{r^{N-2+\beta}}\int_{\Omega_r^+}|\nabla u |^2 \;dx&\leq& \Bigg(\frac{1}{r^\beta}\int_{\Omega_r^+} \frac{|\nabla u |^2}{|x-x_0|^{N-2}} \;dx \Bigg)+ C_0 r^{2-\beta} \notag\\
&\leq&  \Bigg(\frac{1}{r_0^\beta}\int_{\Omega_{r_0}^+} \frac{|\nabla u |^2}{|x-x_0|^{N-2}} \;dx \Bigg)+ C_0 r_0^{2-\beta},
\end{eqnarray}
and we conclude that for every $r\leq r_0/2$,
$$\int_{\Omega_r^+}|\nabla u |^2 \;dx\leq K r^{N-2+\beta},$$
with 
$$K=\Bigg(\frac{1}{(r_0/2)^\beta}\int_{\Omega_{r_0/2}^+} \frac{|\nabla u |^2}{|x-x_0|^{N-2}} \;dx \Bigg)+ C_0 (r_0/2)^{2-\beta}.$$
Let us now provide an estimate on $K$.  By using equations~\eqref{e:beta}  and~\eqref{constantC} and Proposition \ref{propinfty1} we get
\begin{eqnarray}
C_0\leq \frac{2 C (N)}{\eta}\lambda \|u\|_{\infty}^2\leq C(N)  \eta^{-1} \lambda^{1+\frac{N}{2}} \|u\|_{L^2(\Omega)}^2 \label{proptoy1}.
\end{eqnarray}
To estimate the first term in $K$ we use \eqref{boundM} and the fact that $\beta=2-\eta$ to obtain
\begin{eqnarray}
\frac{1}{(r_0/2)^\beta}\int_{\Omega_{r_0/2}^+} \frac{|\nabla u |^2}{|x-x_0|^{N-2}} \;dx \leq  C(N,r_0,\eta)  \lambda(1+\lambda^{\frac{N}{2}})\|u\|_2^2. \label{proptoy2}
\end{eqnarray}
Finally, by combining~\eqref{proptoy1} and \eqref{proptoy2} we obtain \eqref{dec00} and this concludes the proof.
\end{proof}

\begin{remark} \label{grisvard} By relying on an argument similar, but easier, to the one used in the proofs of Lemma \ref{monot} and Proposition \ref{propdecayDir}, 
we get the following estimate. Let $\eta$, $r_0$, $\Omega$ and $x_0$ be as in the statement of Proposition~\ref{propdecayDir}, and assume that $u$ is an harmonic function in $\Omega$ {satisfying homogeneous Dirichlet conditions on the boundary $\partial \Omega\cap B(x_0,r)$}. For every $\eta \in ]0, 1[$ we have  
\begin{equation}
\label{e:decay2}
   \int_{B(x_0,r)\cap \Omega}|\nabla u|^2\dd x \leq
            C r^{N-\eta}  \|u\|_2^2 \; \quad \forall \; r \in ]0,r_0/2[ \, ,
\end{equation}
for a suitable positive constant $C= C(N, r_0,\eta)$.

Note that the decay estimate~\eqref{e:decay2} is sharp, in the sense that we cannot take $\eta =0$, as the following example shows. Let $N=2$ and let $\Omega$ be an angular sector with opening angle $\omega$,
$$
   \Omega: = \Big\{ (r, \theta): \; - \omega/2 < \theta < \omega /2 \Big\}. 
$$
By recalling that the Laplacian in polar coordinates is$$\Delta=\frac{\partial^2}{\partial r}+\frac{1}{r}\frac{\partial}{\partial r}+\frac{1}{r^2}\frac{\partial^2}{\partial \theta},$$
we get that the function $u(r,\theta):= r^{d} \cos(\theta \pi / \omega )$ satisfies the homogeneous Dirichlet condition on $\partial \Omega$ and is harmonic provided that $d=\pi / \omega$. Also, by computing the gradient in the circular Frenet basis $(\tau,\nu)$  and by using formulas
$$\frac{\partial u}{\partial \tau}=\frac{1}{r}\frac{\partial u}{\partial \theta}=-r^{\frac{\pi}{\omega}-1}\frac{\pi}{\omega}\sin(\theta \frac{\pi}{\omega})\quad \quad \text{ and }\quad \quad \frac{\partial u}{\partial \nu}=\frac{\partial u}{\partial r} = \frac{\pi}{\omega} r^{\frac{\pi}{\omega}-1}\cos(\theta \frac{\pi}{\omega}),$$
we get
$$\int_{B(0,r)\cap \Omega }|\nabla u|^2=\int_{B(0,r)\cap \Omega}\Big|\frac{\partial u}{\partial \tau}\Big|^2+\Big|\frac{\partial u}{\partial \nu}\Big|^2=\left(\frac{\pi}{\omega}\right)^2\int_{0}^{r}\int_{-\omega/2}^{\omega/2}s^{2(\frac{\pi}{\omega}-1)} s \;dsd\theta=\frac{\pi}{2} r^{2\frac{\pi}{\omega}},$$
which leads to the following remarks:
\begin{enumerate}
\item[\rm{(i)}] if $\Omega$ is the half-space (i.e., if $\omega=\pi$), then the Dirichlet integral decays like $r^2=r^N$. 
\item[\rm{(ii)}] If the opening angle $\omega< \pi$, then we have a good decay of the order $r^\alpha$ for some $\alpha>2$.
\item[\rm{(iii)}] The most interesting behavior occurs when the opening angle $\omega> \pi$, then the Dirichlet integral decays like $r^{\alpha}$ with $\alpha<2$. Note, moreover, that if we want that $\alpha$ gets closer and closer to $2$, we have to choose 
$\omega$ closer and closer to $\pi$ and this amounts to require that $\Omega$ is an $(\varepsilon, r_0)-$Reifenberg flat domain for a smaller and smaller value of $\varepsilon$.  
\end{enumerate}
We refer to the book of Grisvard \cite{g} for an extended discussion about elliptic problems in polygonal domains.
\end{remark}

\subsubsection{Difference between the projection of eigenfunctions }
By relying on the analyis in~\cite{lm2} we get the following result. Note that we use the convention of identifying any given $u \in H^1_0(\Omega)$ with the function defined on the whole $\R^N$ by setting $u=0$ outside $\Omega$.

\begin{prop} 
\label{estimproj} For any $\alpha \in \, ]0, 1[$ there is threshold $\varepsilon (\alpha) \in \, ]0, 1/2[$ such that the following holds. 
Let $\Omega_a$ and $\Omega_b$ be two $(\varepsilon,r_0)$-Reifenberg-flat domains in $\R^N$, both contained in the disk $D$, which has radius $R$. If 
\begin{eqnarray}
   d_H( \Omega^c_a, \Omega^c_b)\leq \delta  \leq \frac{r_0}{8 \sqrt{N}},  \label{disthaus}
\end{eqnarray}
then there is a constant $ C=  C (N, r_0, \alpha)$ such that, if $u \in H^1_0(\Omega_b)$ is a Dirichlet eigenfunction associated with the eigenvalue $\lambda$, then 
\begin{equation}
\label{e:prest}
 \|\nabla u- \nabla P^{\mathcal D}_{\Omega_a}(u) \|^2_{L^2 (D)} \leq 
     C \lambda(1+\lambda^{\frac{N}{2}}) \delta^{\alpha} L  \; \| u \|^2_{L^2(\Omega_b)},  
\end{equation}
where $L:= \Hh^{N-1}(\partial \Omega_b) $ and the projection $P^{\mathcal D}_{\Omega_a}$ is defined by~\eqref{e:pda} .
\end{prop}
\begin{proof}  We proceed in two steps.\\
\noindent {\sc $\diamond$  Step 1.} We first fix $u \in H^1_0(\Omega_b)$ and construct $\tilde u \in H^1_0(\Omega_a)$ which is ``close" to $u$, in the sense specified in the following. 
 
{To begin with, we point out that \eqref{disthaus} implies
that 
$\left\{ B(x, 2 \delta) \right\}_{x \in \partial \Omega_b} $ is a covering 
of $\Omega_b \setminus \Omega_a$. Indeed, by contradiction assume there is $y \in \Omega_b \setminus \Omega_a$ such that 
$$
    2 \delta < d (y, \partial \Omega_b) = d(y, \Omega_b^c) \leq \sup_{y \in \Omega_a^c} d (y, \Omega_b^c). 
$$
This would contradict \eqref{disthaus} and hence the implication holds true.}  

By applying Lemma~\ref{cov1} with $r=5\delta /2$ we can find a finite set $I$, such that $\sharp I \leq C(N) L / \delta^{N-1}$ and $\{ B(x_i, 5\delta/2 \}_{i \in I}$ is a covering of $\Omega_b \setminus \Omega_a$. 

Next, we use the function $\theta_0$ given by  Lemma \ref{thetazero} (with $r=5\delta/2$)    and    we set $\tilde u(x) : = \theta_0 (x) u(x)$. We observe that, since 
$$
    \Omega_a^c \subseteq \Omega_b^c \cup \big( \Omega_b \setminus \Omega_a \big)
    \subseteq \Omega_b^c \cup \bigcup_{i \in I} B \left( x_i, 5\delta / 2 \right), 
$$
then $\tilde u \in H^1_0 (\Omega_a)$. Also, ${\nabla \tilde u = u \nabla \theta_0 + \theta_0 \nabla u}$ and hence 
\begin{equation}
\label{e:nablav}
    \| \nabla \tilde u - \nabla u \|_{L^2(D)} \leq 
    \| (1- \theta_0) \nabla u \|_{L^2 (D)} + \| u \nabla \theta_0 \|_{L^2(D)} 
 \end{equation}
Next, by recalling that $\theta_0 \equiv 1$ outside the union of the balls $\{ B(x_i, 5 \delta ) \}_{i \in I}$, we get 
\begin{equation}
\label{e:nablav3} 
          \| (1- \theta_0) \nabla u \|^2_{L^2 (D)} =
          \int_{\bigcup_{i \in I} B(x_i, 5 \delta)} (1 - \theta_0)^2 |\nabla u|^2 dx \leq 
          \sum_{i \in I}   \int_{ B(x_i, 5 \delta)}  |\nabla u|^2  dx. 
\end{equation}
Also, by recalling that $|\nabla \theta_0| (x) \leq C(N)/\delta$, we obtain 
\begin{equation}
\label{e:nablav2}
\begin{split}
         \| u \nabla \theta_0 \|^2_{L^2(D)} & =
         \int_{\bigcup_i B(x_i, 5 \delta)} |\nabla \theta_0|^2  u^2dx 
          \leq 
         \frac{C(N)}{\delta^2} \sum_i \int_{B(x_i, 5 \delta) }u^2dx \\ & \leq 
         \frac{C(N)}{\delta^2} \sum_i  (5 \delta)^2 \int_{B(x_i, 5\sqrt{N} \delta) }
         |\nabla u|^2 dx. \\
     \end{split}
\end{equation}
To get the last inequality we have used~\cite[Proposition 12]{lm2} and the fact that one can take $b(N) = \sqrt{N}$ in there. \\
\noindent {\sc $\diamond$  Step 2.} We now restrict to the case when $u$ is an eigenfunction for the Dirichlet Laplacian, and $\lambda$ is the associated eigenvalue. By using Proposition~\ref{propdecayDir} and Lemma~\ref{cov1} we get   
\begin{equation}
\label{e:ref}
    \sum_i \int_{B(x_i, 4 \sqrt{N} \delta) }
         |\nabla u|^2 dx \leq  \sum_i C \lambda(1+\lambda^{\frac{N}{2}})\delta^{N-\eta} \|u\|_2^2 \leq  C\lambda(1+\lambda^{\frac{N}{2}}) L \delta^{1-\eta} \|u\|_2^2, 
\end{equation}
with $C=C(N,  r_0,\eta)$ and provided that $\delta \leq r_0 / 8 \sqrt{N}$. 

By choosing $\eta: = (1- \alpha)$, inserting~\eqref{e:ref} into~\eqref{e:nablav3} and\eqref{e:nablav2} and recalling~\eqref{e:nablav} we finally get 
$$
    \|\nabla u-\nabla P^{\mathcal D}_{\Omega_a}(u) \|^2_{L^2 (D)}  \leq 
    \|\nabla u - \nabla \tilde u \|_{L^2(D)} \leq 
    C(N, r_0,\alpha) \lambda(1+\lambda^{\frac{N}{2}})L \delta^{\alpha} \|u\|_2^2 .
 $$ 
\end{proof}
\subsection{Conclusion of the proof of Theorem~\ref{main}}
\begin{proof}[Proof of Theorem \ref{main}] 
First,  we recall that $\max \{ \lambda_n^a, \lambda_n^b \} \leq \gamma_n$.  Next, we fix $n \in \mathbb N$ and we denote by $u^b_1, \dots, u^b_n$ the first $n$ eigenfunctions of the Dirichlet Laplacian in $H^1_0(\Omega_b)$. {Given $u = \sum_{k=1}^n c_k u_b^k$, by applying Proposition~ \ref{estimproj}  we get 
\begin{equation}
\label{e:space}
\begin{split}
       \| \nabla u - \nabla P^{\mathcal D}_{\Omega_a} u  \|^2_{L^2(D)} & = 
\left\| \sum_{k=1}^n c_k \Big( \nabla u_b^k - \nabla  P^{\mathcal D}_{\Omega_a} u^k_b  
\Big) \right\|^2_{L^2(D)} \leq n \sum_{k=1}^n c^2_k \left\| \nabla u_b^k - \nabla  P^{\mathcal D}_{\Omega_a} u^k_b  \right\|^2_{L^2(D)}   \\
      & \leq  n \sum_{k=1}^n c^2_k 
      {C\gamma_n(1+\gamma_n^{\frac{N}{2}})L\delta^{\alpha}} \| u^k_b  \|^2_{L^2(D)}   =  n {C\gamma_n(1+\gamma_n^{\frac{N}{2}})L\delta^{\alpha}} \| u \|^2_{L^2 (D)}.  
      \\
\end{split}
\end{equation}}
To get the last equality we have used the fact that the eigenfunctions associated with different eigenvalues are orthogonal with respect to the standard scalar product in $L^2$. 

We use~\eqref{e:space} and the Sobolev-Poincar\'e inequality in the ball $D$, which has radius $R$, and we conclude that the assumptions of Lemma \ref{l:ab} are verified provided  that: \\
${\rm (i)} \; A=nC\gamma_n(1+\gamma_n^{\frac{N}{2}})  L \delta^\alpha$, $B=C(N, R) A$ and \\$\rm{(ii)}$ $B<1$.\\
 Hence, we fix a threshold $\delta_0 (N, n, \gamma_n, r_0, \alpha, R, L)$ satisfying $n C\gamma_n(1+\gamma_n^{\frac{N}{2}})  L \delta_0^\alpha <1/4$, and we get that for any $\delta \leq \delta_0$ one has $1/(1-\sqrt{B})\leq 2$. By applying Lemma \ref{l:ab} we then  get
$$
    \lambda_n^a-\lambda_n^b \leq \frac{A}{(1-\sqrt{B})}\leq 2A 
     \leq  2 n C\gamma_n(1+\gamma_n^{\frac{N}{2}})   L      \delta^\alpha. 
$$
The theorem follows by exchanging the roles of $\Omega_a$ and $\Omega_b$.
\end{proof}
\subsection{The Lipschitz case}
\label{Lip}
We end this section with the proof of Theorem \ref{mainLip}.

\begin{proof}[Proof of Theorem \ref{mainLip}] Let $\Omega_a$ and $\Omega_b$ be as in the statement of Theorem~\ref{mainLip}. We fix $n \in \mathbb N$ and $k \leq n$ and we denote by $u^k_a \in H^1_0(\Omega_a)$ and $u^k_b \in H^1_0(\Omega_b)$ the eigenfunctions  associated with $\lambda_a^k$ and $\lambda_b^k$, respectively. In particular, $u^k_b$ solves $-\Delta u=\lambda_b^k u^k_b$ in $\Omega_b$. Let $\bar u_b \in H^1_0(\Omega_a)$ be the distributional solution of 
\begin{equation*}
\left\{
\begin{array}{ll}
   -\Delta u = \lambda^k_b u_b & \text{in} \; \; \Omega_a \\
   u =0 & \text{on} \; \partial \Omega_a.
\end{array}
\right.
\end{equation*}
We can now apply Theorem 1  in the paper by Savar\'e and Schimperna~\cite{SavareSchimperna}. Then formula (3.4) in \cite{SavareSchimperna} yields
$$
   \| \nabla u^k_b - \nabla  P^{\mathcal D}_{\Omega_a} u^k_b \|^2_{L^2(D)} \leq 
   \| \nabla u^k_b - \nabla \bar u_b  \|_{L^2(D)} \leq 
   C( \rho, N, {R}) \| \lambda_b^k u^k_b \|_{L^2(D)} \| \lambda_b^k u^k_b \|_{H^{-1} (D)} 
   \frac{d_H (\Omega_a^c, \Omega_b^c)}{\rho \sin \theta} ,
$$
where the projection $P^{\mathcal D}_{\Omega_a} $ is the same as in~\eqref{e:pda}. { Then, by using the definition of eigenfunction, we get 
\begin{equation*}
\begin{split}
   \|\lambda_b^k u^k_b \|_{H^{-1} (D)} & \leq  C(N, R) \sqrt{\lambda_b^k }\| u^k_b\|_{L^2(D)}.     \\
   \end{split}
\end{equation*}
}
We recall that $\lambda^k_b \leq \gamma_n$ and that $S^n_b$ is the eigenspace of $H^1_0(\Omega_b)$ generated by the first $n$ eigenfunctions and by arguing as~\eqref{e:space}, we get that, for any $u \in S^n_b$, 
$$
      \| \nabla u - \nabla  P^{\mathcal D}_{\Omega_a} u \|^2_{L^2(D)} \leq 
      C(\theta, \rho, \gamma_n, N, {R}, n) d_H (\Omega_a^c, \Omega_b^c) \| u \|^2_{L^2(D)}. 
$$
Hence, by proceeding as in the proof of Theorem~\ref{main} we can conclude.  
 \end{proof}
\section{Stability estimates for Neumann eigenvalues}
\label{s:proofmain2}
\subsection{The Neumann eigenvalue problem in Reifenberg-flat domains}
\label{ss:nep}
Given an open and bounded set $\Omega$, $(u, \mu) \in H^1 (\Omega) \setminus \{ 0\} \times \mathbb R$ is an eigencouple for the Neumann Laplacian in $\Omega$ if
\begin{equation}
\label{e:dpw2}
          \int_{\Omega} \nabla u (x) \cdot \nabla v (x) dx = \mu \int_{\Omega}
         u (x) \, v (x) dx \qquad \forall \,  v \in H^1 (\Omega).
\end{equation}
In this section, we apply the abstract framework introduced in Section~\ref{s:estpro} to study how the eigenvalues $\mu$ satisfying~\eqref{e:dpw2} depend on the domain $\Omega$. We set 
$$
    H = L^2 (\R^N,\R) \times L^2 (\R^N,\R^N) 
$$
and we equip it with the scalar product
\begin{equation}
\label{e:accan}
    \h \Big( (u_1,  v_1), (u_2,  v_2) \Big) = 
    \int_{\R^N} u_1 (x) \, u_2 (x) dx +   \int_{\R^N} v_1 (x) \cdot  v_2 (x) dx.
\end{equation}
Also, we set 
\begin{equation}
\label{e:hn}
    h \Big( (u_1,  v_1), (u_2,  v_2) \Big) = 
    \int_{\R^N} u_1 (x) \, u_2 (x) dx. 
\end{equation}
Note that $h$ is a symmetric, positive bilinear form (i.e., it satisfies properties~\eqref{e:c1}), although it is not a scalar product on $H$.  Inequality~\eqref{e:poinab} is trivially satisfied. 

As before, $\Omega_a$ and $\Omega_b$ are two Reifenberg-flat domains contained in $\mathbb R^N$ and we denote the Sobolev spaces by $H^1 (\Omega_a)$ and $H^1 (\Omega_b)$. The spaces $V_a$ and $V_b$ are defined by considering the map
\begin{equation*}
  \begin{split}
        j_{\Omega}:  H^1 ( \Omega & )  \to    L^2 (\R^N) \times L^2 (\R^N, \R^N)  \\
        & u \mapsto (u {\bf 1}_\Omega, \nabla u {\bf 1}_\Omega), \\
       \end{split}      
\end{equation*}       
where ${\bf 1}_\Omega$ denotes the characteristic function of $\Omega$. Note that the ranges $V_a = j_{\Omega_a}(H^1 (\Omega_a))$ and $V_b = j_{\Omega_b}(H^1 (\Omega_b))$ are closed and that~\eqref{assumption3} is satisfied. Note also that  the inclusion~\eqref{e:rellichab} {is compact} because, in virtue of Proposition~\ref{embedding}, we can apply Rellich's Theorem. 

The Neumann problem~\eqref{e:dpw2} reduces to~\eqref{e:eigpb} provided that $\lambda= \mu + 1$ and hence it is solved by a sequence of eigencouples $(\mu_n, u_n)$ with $\lim_{n \to + \infty}\mu_n = + \infty$ . By relying on Lemma~\ref{l:ab}, we deduce that to control the difference $|\mu^a_n - \mu^n_b|$ it is sufficient to provide an estimate on the projection operator defined by~\eqref{e:proj}, which can be identified by a map $P^{\mathcal{N}}_{\Omega_a}: H \to H^1 (\Omega_a)$ satisfying the following property:  $\forall \; u \in H^1 (\Omega_b),$
\begin{eqnarray}
    \| j_{\Omega_b}(u) -P^{\mathcal{N}}_{\Omega_a} j_{\Omega_b}(u) \|_{\mathcal{H}}^2   &=& \big\|  {\bf 1}_{\Omega_b} u - {\bf 1}_{\Omega_a} P^{\mathcal{N}}_{\Omega_a} u  \big\|_{L^2 (\R^N)}^2  + 
            \big\| {\bf 1}_{\Omega_b} \nabla u-  {\bf 1}_{\Omega_a} \nabla  [ P^N_{\Omega_a} u ] \big\|_{L^2 (\R^N)}^2    \notag \\
&=&        \min_{v \in H^1 (\Omega_a)} \Big\{ \big\| {\bf 1}_{\Omega_b} u - {\bf 1}_{\Omega_a} v  \big\|_{L^2 (\R^N)}^2  +  \big\| {\bf 1}_{\Omega_b} \nabla u-  {\bf 1}_{\Omega_a} \nabla v \big\|_{L^2 (\R^N)}^2  \Big\}. \label{e:projneu}
\end{eqnarray}
\subsection{Estimates on the projection of Neumann eigenfunctions}
To provide an estimate on~\eqref{e:projneu} we first establish a preliminary result concerning the decay of the gradient of a
Neumann eigenfunction. 
\begin{prop}
\label{propdecayneum} 
For every $\eta>0$ there is a positive constant $\varepsilon:=\varepsilon(\eta)$ such that, for every  {connected}, $(\varepsilon, r_0)$-Reifenberg-flat domain $\Omega \subseteq \R^N$, the following holds. Let $u$ be a Neumann eigenfunction in $\Omega$ associated with the eigenvalue $\mu$, let $x\in \partial \Omega$ and let $r\leq 
\min \{ r_0,1 \}$. Then there is a constant $C = C (N, r_0, \eta, \diam(\Omega))$ such that 
\begin{eqnarray}
\label{e:bdecay}
           \int_{B(x,r)\cap \Omega}|\nabla u|^2 \;dx \leq
            C \mu  (1+\sqrt{\mu})^{2\gamma(N)} \|u\|^2_{L^2(\Omega)} \left(\frac{r}{\min \{ r_0,1 \} }\right)^{N-\eta}, \label{dec}
\end{eqnarray}
where $\gamma (N) = \max \Big\{ \frac{N}{2}, \frac{2}{N-1} \Big\}$ as in the statement of Proposition~\ref{uestim}. 
\end{prop}

\begin{proof} 
For a given $\eta>0$  we choose $\beta$ in such a way that $0<\beta<\eta$ and that
\begin{eqnarray}
\label{defina}
a := 8^{\frac{1}{\beta-\eta}}<\frac{1}{2} . 
\end{eqnarray}
Also, we choose $\varepsilon $ smaller or equal to the constant given by Theorem~\ref{thdecay} with this choice of $a$ and  $\beta$. Note that $\varepsilon$ only depends on $\eta$.

We now consider a Neumann eigencouple $(u, \mu)$, while the point $x\in\partial \Omega$ is fixed. We may assume without losing  generality that $r_0\leq 1$, up to redefine $r_0$ by $\min \{1,r_0 \}$.

We first use  the induction principle in order to show that, 
for a suitable constant $C_4 = C_4 (N, r_0, \diam(\Omega))$ that will be chosen later, and for any $k \in \N$,
\begin{equation}
\label{Inductivedecay}
                     \int_{B(x,a^{k}r_0)\cap \Omega}|\nabla u|^2dx\leq C_4\mu  (1+\sqrt{\mu})^{2\gamma(N)}  a^{k(N-\eta)}  \| u \|^2_{L^2(\Omega)}. 
\end{equation}
If $k=0$ the inequality~\eqref{Inductivedecay} is satisfied provided that $C_4\geq 1$ because $\int_{\Omega}|\nabla u|^2dx = \mu \|u\|_{2}^2$. Next, we consider the inductive step and we assume that \eqref{Inductivedecay} holds for a given $k\geq 0$. We term $v$ the solution of Problem~\eqref{probbb} in $\Omega_k:= B(x,a^kr_0)\cap \Omega$. Then Theorem \ref{thdecay} gives  
\begin{eqnarray}
\int_{\Omega_{k+1}}|\nabla v |^2 dx \leq  a^{N-\beta}
\int_{\Omega_k}|\nabla v|^2dx .\label{decay5}
\end{eqnarray}
Note that since $a<1$, then $\Omega_{k+1} \subseteq \Omega_k$. We now compare $\nabla u$ and $\nabla v$ in $B(x,a^{k+1}r_0)$ by using the inequality $\|a+b\|^2\leq 2\|a\|^2 +2 \|b\|^2$ : 
\begin{eqnarray}
\int_{\Omega_{k+1}}|\nabla u|^2 dx  & \leq &
 2 \int_{\Omega_{k+1}}|\nabla v|^2 dx  + 2 \int_{\Omega_{k+1}}|\nabla (u-v)|^2  dx \notag \\
&\leq &  2 a^{N-\beta}
\int_{\Omega_{k}}|\nabla v|^2dx + 2 \int_{\Omega_k}|\nabla (u-v)|^2 dx  \notag \\
&\leq &  2 a^{N-\beta}
\int_{\Omega_k}|\nabla u|^2dx +  2 \int_{\Omega_k}|\nabla (u-v)|^2 dx .\label{etape1}
\end{eqnarray}
Since $v$ is harmonic and $u$ is a competitor, then $\nabla (u-v)$ is orthogonal to $\nabla v$ in $L^2(\Omega_k)$ and hence
\begin{equation}
\label{popol}
       \int_{\Omega_k}|\nabla (u-v)|^2 dx = \int_{\Omega_k}|\nabla u|^2 dx - \int_{\Omega_k}|\nabla v|^2 dx.
\end{equation}
Moreover,  $u$ minimizes the functional $\displaystyle{w\mapsto \int_{\Omega_k} | \nabla w |^2 dx - 2 \mu \int_{\Omega_k}w u \, dx }$ with its own Dirichlet conditions on $\partial B(x,a^kr_0)\cap \Omega$, and hence by taking $v$ as a competitor we obtain
\begin{eqnarray}
\int_{\Omega_k}|\nabla u|^2 dx -\int_{\Omega_k}|\nabla v|^2  dx &\leq& 2\mu \int_{\Omega_k} \big[ u^2 -vu \big] dx \notag \\
&\leq& 4\mu \omega_Na^{Nk} \|u\|_{\infty}^2 . \label{norminf}
\end{eqnarray}
To get the above inequality we have used that $r_0\leq 1$ and the estimate
\begin{equation}
\label{e:esti}
     \| v \|_{L^{\infty} (\Omega)} \leq \| u \|_{L^{\infty} (\Omega)},
\end{equation}
which can be established arguing by contradiction. Indeed, set $M : = \| u \|_{L^{\infty} (\Omega)}$. If~\eqref{e:esti} is violated, the truncated function 
$w: = \min \{ M, \, \max \{ v, -M \} \}$ would be a competitor of $v$ satisfying $\| \nabla w \|^2_{L^2 (\Omega)} < \| \nabla v \|^2_{L^{2} (\Omega)}$, which contradicts the definition of $v$.  

By plugging~\eqref{popol} and~\eqref{norminf} in \eqref{etape1} and using the inductive hypothesis \eqref{Inductivedecay}, we get the estimate
$$\int_{\Omega_{k+1}}|\nabla u|^2 \leq 2 a^{N-\beta}
C_4\mu  (1+\sqrt{\mu})^{2\gamma(N)} a^{k(N-\eta)}\|u\|_2^2+ 8 \mu \omega_Na^{Nk} \|u\|_{L^\infty(\Omega)}^2,$$
where $\omega_N$ is the measure of the unit ball in $\R^N$. By using \eqref{desirein}, the above expression reduces to 
\begin{eqnarray}
\int_{\Omega_{k+1}}|\nabla u|^2 \leq\mu   (1+\sqrt{\mu})^{2\gamma(N)} a^{(k+1)(N-\eta)} \|u\|_2^2\Big[2C_4 a^{\eta-\beta}+C_5a^{\eta(k+1)-N}\Big] , \label{inde}
\end{eqnarray}
for some constant $C_5= C_5 (N, r_0, \diam(\Omega))$. We claim that by choosing in~\eqref{Inductivedecay} 
\begin{eqnarray}
C_4= 8 C_5 a^{\eta-N}\label{totol}
\end{eqnarray}
then the right hand side of~\eqref{inde} is less than $C_4 \mu  (1+\sqrt{\mu})^{2\gamma(N)} a^{(k+1)(N-\eta)}$, which proves \eqref{Inductivedecay}. Indeed, our choice of $a$ implies $a^{\eta-\beta}=\frac{1}{8}$ and hence~\eqref{totol} implies 
\begin{eqnarray}
\ \ \ \ \ \ 2 C_4a^{\eta-\beta}+ C_5 a^{(k+1) \eta - N }= C_4\frac{2}{8} +  (8C_5 a^{\eta-N})\frac{a^{k\eta}}{8} \leq C_4\frac{2}{8} +  (8C_5 a^{\eta-N})\frac{1}{8} = \frac{3}{8}C_4, \label{garen}
\end{eqnarray}
which concludes the proof of~\eqref{Inductivedecay}. 

To conclude the proof of the Proposition we observe that, given $r\leq r_0$, we can select an integer $k\geq 0$ such that 
$r <  a^{k}r_0 \leq r a^{-1},$
which yields
\begin{equation*}
\begin{split}
     \int_{B(x,r)  \cap \Omega}|\nabla u|^2 \dx \leq  &
     \int_{B(x,a^{k}r_0) \cap  \Omega} |\nabla u|^2 \dx\leq 
    C_4 \mu  (1+\sqrt{\mu})^{2\gamma(N)}  a^{k(N-\eta)} \| u\|^2_{L^2 (\Omega) }  \\ & 
    \leq C_4 \mu  (1+\sqrt{\mu})^{2\gamma(N)}  \| u\|^2_{L^2 (\Omega) } 
   \left( \frac{r}{ar_0 }\right)^{N-\eta} \\
\end{split}
\end{equation*}
and this implies~\eqref{dec} provided that $C:=C_4 / a^{N - \eta}$.
\end{proof}
By combining a covering argument from~\cite{lm} with the previous proposition we establish the projection estimate provided by the following result.
\begin{prop}\label{estimproj2} 
For any $\alpha \in \, ]0, 1[ $ there is a constant $\varepsilon = \varepsilon (\alpha)\leq 1/600$ such that the following holds. 
Let $\Omega_a$ and $\Omega_b$ be two {connected} $(\delta,r_0)$-Reifenberg-flat domains of $\R^N$ satisfying
\begin{equation}
\label{e:boundsdh}
  \max \{ d_H( \Omega_a,  \Omega_b);  d_H( \Omega^c_a,\Omega^c_b) \} \leq \delta, 
\end{equation}
where $0< \delta \leq \min \{r_0/ 5, 1\}$. Then there is a constant $C=C(N,r_0, \diam(\Omega_b), \alpha)$ such that, if $u \in H^1(\Omega_b)$ is a Neumann eigenvector associated with the eigenvalue $\mu$, then there is $\tilde u \in H^1 (\Omega_a)$ satisfying 
$$
 \| {\bold 1}_{\Omega_a} \tilde u  - {\bold 1}_{\Omega_b}u  \|_{L^2(\R^N)}^2
 + \| {\bold 1}_{\Omega_a} \nabla \tilde u  - {\bold 1}_{\Omega_b} \nabla u  \|_{L^2(\R^N)}^2
  \leq  C  (1+\sqrt{\mu})^{2 \gamma(N)+1}   L \delta^\alpha  \| u\|^2_{L^2(\Omega_b)},
$$
where $L:= \Hh^{N-1}(\partial \Omega_b)$ and $\gamma (N) = \max {\displaystyle \left\{ \frac{N}{2}, \frac{2}{N-1}\right\}}$ as in the statement of Proposition~\ref{uestim}.
\end{prop}
\begin{proof}
The goal of the first part of the proof is to construct a function $\tilde u$ which only differs from $u$ in a narrow strip close to the boundary: this is done by relying on a covering argument similar to those in~\cite{lm} (see Lemma 9 in there). The proof is then concluded by relying on Theorem~\ref{uestim} and Proposition~\ref{propdecayneum}.  The details are organized in the following steps. \\
{\sc $\diamond$ Step 1.} We construct a partition of unity. First, we observe that~\eqref{e:boundsdh} implies that $\{ B(x, 2 \delta) \}_{x \in \partial \Omega_b}$ is a covering of $\Omega_a \triangle \Omega_b$ and by relying on Lemma~\ref{cov1} we can find a finite set $I$, such that (i) $\sharp I \leq C (N)L / \delta^{N-1}$; (ii) $\{ B(x_i, 5\delta/2 ) \}_{i \in I}$ is a covering of $\Omega_a \triangle \Omega_b$. 
To simplify the exposition, in the following we use the notations
$$
   B_i: = B(x_i, 5\delta/2 ) \qquad 2 B_i : = B(x_i, 5\delta ) \qquad 6 B_i : = B(x_i, 15 \delta )  \qquad 
   W = \bigcup_{i=1}^{\sharp I} 2 B_i.
$$
Note that property~\eqref{e:disjoint} in the statement of Lemma~\ref{cov1} implies that, for every $x \in \R^N$, $\sharp \{ i : \; x \in  2B_i \} \leq C (N)$ and hence, in particular, that, for every integrable function $v$,  
\begin{equation}
\label{e:bdcover}
        \sum_{i=1}^{\sharp I} \int_{2 B_i} |v (x)| dx \leq C (N) \int_{\bigcup_i 2 B_i}  |v(x)| dx. 
\end{equation}
Next, we apply Lemma~\ref{thetazero} with $r:= 5 \delta/2$ and we obtain Lipschtz continuous functions $\theta_0, \theta_1, \dots, \theta_{\sharp I}:  \R^N \to [0, 1]$ satisfying {
\begin{equation}
\label{e:thetas} 
\begin{split}
&     |\nabla \theta_i (x) | \leq \frac{C(N)}{\delta} \; a.e. \, x \in \R^N, \quad i= 0, \dots, \sharp I \\
&        \theta_0  (x) =0  \; \; \mathrm{if} \;
        x \in \bigcup_{i \in I} B_i, \qquad \quad
        \theta_0  (x) =1\; \; \mathrm{if} \;
        x \in  \R^N \setminus \bigcup_{i \in I} 2 B_i \\
        &         \theta_i  (x) =0 \; \; \mathrm{if} \; x \in  \R^N \setminus  2 B_i, \quad     
         i=1, \dots, \sharp I, \qquad \quad \sum_{i=0}^{\sharp I} \theta_i(x) =1 
        \; \; \text{for every}\;  x \in \R^N.  \\
  \end{split}
  \end{equation}}

\noindent {\sc $\diamond$ Step 2.}  We define the function $\tilde u$.    For $i \in I$ we term $Y_i$ the point in $\Omega_b\cap 2 B_i$ such that $d(Y_i, x_i )=3\delta $ and  the vector $Y_i-x_i$ is orthogonal to $P(x_i,5\delta)$. Note that such a point exists provided $5 \delta \leq r$ and $\varepsilon \leq 3/10$ {due to Lemma 5 in~\cite{GeomPart} (see also Proposition 2.2 in \cite{hm3}}). Then we define the domain $D_i:=B(Y_i,\delta)\subseteq \Omega_b\cap 2 B_i$ and we set 
$$
    m_i : = \frac{1}{|D_i |} \int_{D_i} u(x) dx \qquad i=1, \dots, \sharp I.
$$
Note that $({\bold 1}_{\Omega_b}u(x))\theta_0(x)$ is well defined for $x\in \R^N$ and belongs to $H^{1}(\R^N)$ because $\theta_0(x)=0$ in a neighborhood of $\partial \Omega_b$. This allows us to define

 \begin{equation}
 \tilde u (x) : =    ({\bold 1}_{\Omega_b} u(x))\theta_0(x) + \sum_{i=1}^{\sharp I} m_i \theta_i (x) \quad \forall x \in   \R^N, \notag
 \end{equation}
so that $\tilde u \in H^1(\Omega_a)$. 

{\sc $\diamond$ Step 3.}  We provide an estimate on $\| \tilde u  {\bold 1}_{\Omega_a} - u {\bold 1}_{\Omega_b} \|^2_{L^2 (\R^N)}$. { First, we point out that $\Omega_a \triangle \Omega_B \subseteq W$ and that  
$\tilde u  {\bold 1}_{\Omega_b} = u {\bold 1}_{\Omega_a}$ in $\R^N \setminus W$, so by recalling the definition of $\tilde u$ we have 
\begin{equation*}
\begin{split}
         \| 
         \tilde u  {\bold 1}_{\Omega_a} - u {\bold 1}_{\Omega_b} 
         \|^2_{L^2 (\R^N)} = &
         \int_{ \Omega_a \cap \Omega_b \cap W} |\tilde u - u|^2 (x) dx +
         \int_{ \Omega_b \setminus \Omega_a} u^2 (x) dx + 
          \int_{ \Omega_a \setminus \Omega_b} \tilde u^2 (x) dx  \\
          & \leq   2 \int_{W \cap \Omega_b} u^2 (x) dx + 
           2   \int_{W}  \left( \sum_{i=1}^{\sharp I} m_i \theta_i (x) \right)^2 dx 
\end{split}
\end{equation*}
By using the convexity of the square function, Jensen inequality and estimate~\eqref{e:bdcover} we have
\begin{equation*}
\begin{split}
       \int_{W}  \left( \sum_{i=1}^{\sharp I} m_i \theta_i (x) \right)^2 dx  &  \leq
            \int_{W}   \sum_{i=1}^{\sharp I} m_i^2  \theta_i (x)  dx 
            \leq 
           \int_{W}   \sum_{i=1}^{\sharp I}\frac{1}{|D_i|} \left( \int_{D_i} u^2 (y) dy\right) 
            \theta_i (x)  dx \\&  \leq  
          \sum_{i=1}^{\sharp I} \frac{1}{\omega_N \delta^N} \left( \int_{D_i} u^2 (y) dy\right)  \int_{2 B_i}
          \theta_i(x) dx \\
          & \leq   
           \sum_{i=1}^{\sharp I} \frac{1}{\omega_N \delta^N} \left( \int_{2 B_i \cap \Omega_b} u^2 (y) dy\right) 
           \omega_N 5^N \delta^N \leq C(N) \int_{\bigcup_i 2 B_i \cap \Omega_b}  u^2 (y) dy. \\
\end{split}
\end{equation*}}
In the previous expression, $\omega_N$ denotes as usual the Lebesgue measure of the unit ball in $\R^N$. By combining the previous two estimates we conclude that
\begin{equation}
\label{e:chil2}
        \| \tilde u  {\bold 1}_{\Omega_a} - u {\bold 1}_{\Omega_b} 
         \|^2_{L^2 (\R^N)} \leq C(N) \int_{W}  u^2 (y) dy
\end{equation}
{\sc $\diamond$ Step 4.} We introduce some notations we need in {\sc Step 5}. Given $i_0 \in I$ we denote by $J_{i_0}$ be the finite set of indices $j \in I$ such that $2B_{j} \cap 2B_{i_0}\not = \emptyset$. Note that $\sharp J_{i_0} \leq C(N)$  by property~\eqref{e:disjoint} in the statement of Lemma~\ref{cov1} and, also, that any ball $2 B_{j}$ is contained in $6B_{i_0}$ if $j \in J_{i_0}$. 

Let $P_0$ be the hyperplane in $6 B_{i_0}$ provided by the definition of Reifenberg flatness and let $\nu_0$ denote its unit normal vector, oriented in such a way that $x_{i_0} + 15 \delta \nu_0 \in \Omega_b$. Also, let $Y_j$ and $D_j$ be as in {\sc Step 2}, and let $P_j$ denote the hyperplane $P(x_j, 5 \delta)$. For any $j \in J_{i_0}$ we have 
$$
    d_H(P_j  \cap 2B_j,P_0\cap 2B_j)\leq d_H(P_j \cap 2B_j ,\partial \Omega_b\cap 2B_j) + d_H(\partial \Omega_b\cap 2B_j,P_0\cap 2B_j )\leq 5\delta\varepsilon + 15\delta \varepsilon.
$$
Hence, 
$$d(Y_j,P_0)\geq d(Y_j,P_j)-d_H(P_j \cap 2B_j ,P_0\cap 2B_j )\geq 3\delta - 20 \varepsilon \delta\geq 2\delta,$$ 
provided that $\varepsilon \leq 1/20$. 
This shows that 
$$\bigcup_{j \in J_{i_0}} D_j \subseteq \widehat{D}_{i_0} := B(x_{i_0},15\delta) \cap \{ x\; ;\;  (x-x_{i_0}) \cdot \nu_0  \geq \delta \}\subseteq \Omega_b,$$
where the last inclusion holds by {Lemma 5 in~\cite{GeomPart} (see also Proposition 2.2 in \cite{hm3})} since $30 \varepsilon \delta \leq \delta$.
 
The key point in this construction is that $\widehat{D}_{i_0}$ is a Lipschitz domain and satisfies the Poincar\'e-Sobolev inequality with constant $C(N) \delta$. Hence, by setting  
$$ \widehat{m}_{i_0}:=\frac{1}{|\widehat{D}_{i_0}|}\int_{\widehat{D}_{i_0}}u(y) \;dy$$
and by using Jensen inequality we get that, for every $j\in J_{i_0}$, we have 
{
\begin{equation}
\label{mimx}
\begin{split}
|m_j-\widehat{m}_{i_0}|^2 & =
\left( \frac{1}{|D_j|}\int_{D_j}(u (x) -\widehat{m}_{i_0})dx \right)^2 
\leq \frac{1}{|D_j|}\int_{D_j}|u (x) -\widehat{m}_{i_0}|^2 dx \\
&  
\leq \frac{1}{|D_j|}\int_{\widehat{D}_{i_0}}|u (x) -\widehat{m}_{i_0}|^2 dx 
 \leq \frac{C(N) \delta^2}{|D_j|}\int_{\widehat{D}_{i_0}}|\nabla u(x) |^2 dx \\
& 
\leq  \frac{C(N)}{\delta^{N-2}}\int_{6B_{i_0} \cap \Omega_b}|\nabla u(x)|^2 dx. \\
\end{split}
\end{equation}}
On the other hand, by definition of $\widehat D_{i_0}$, we have that, for any $y \in 6B_{i_0}\cap \Omega_b\setminus \widehat D_{i_0}$,   
$$d (y,\partial \Omega_b)\leq d(y,P_0)+d_H(P_0\cap 6B_{i_0}, \partial \Omega_b\cap 6B_{i_0})\leq   \delta + { 6}\varepsilon \delta\leq 2\delta,$$
which implies that $y\in \bigcup_{x\in�\partial \Omega_b}B(x,2\delta)$. In particular, 
$$2B_{i_0}\cap\Omega_b\setminus  \widehat D_{i_0}\subseteq \bigcup_{x\in \partial \Omega_b}B(x,2\delta) \subseteq \bigcup_{i \in I }B_i$$
and this implies that 
$ \mathrm{supp}(\theta_0)\cap 2B_{i_0}\subseteq \widehat D_{i_0}.$ \\
 {\sc $\diamond$ Step 5.}  We provide an estimate on 
$\| {\bf 1}_{\Omega_a}  \nabla \tilde u - {\bf 1}_{\Omega_b}  \nabla u \|_{L^2 (\R^N)}$. 
First, we recall that ${\Omega_a \triangle \Omega_b \subseteq W}$ and we observe that 
\begin{equation}
\label{e:grad}
\begin{split}
    \| {\bf 1}_{\Omega_a}  \nabla \tilde u - {\bf 1}_{\Omega_b}  \nabla u \|^2_{L^2 (\R^N)}  
     &  = 
   \int_{\Omega_a \cap \Omega_b \cap W} | \nabla \tilde u -  \nabla u  |^2 (x) dx +
   \int_{\Omega_a  \setminus \Omega_b}  | \nabla \tilde u  |^2 (x) dx + 
   \int_{\Omega_b  \setminus \Omega_a}  | \nabla  u  |^2 (x) dx \\
   & \leq  
   2  \int_{W\cap \Omega_b}  | \nabla  u  |^2 (x) dx +
   2 \int_{\Omega_a \cap \Omega_b \cap W} | \nabla \tilde u  |^2 (x) dx +
     \int_{\Omega_a  \setminus \Omega_b}  | \nabla \tilde u  |^2 (x) dx .  
\end{split}
\end{equation}
{The first term in the last line of the above expression satisfies
\begin{equation}
\label{e:one}
          2  \int_{W\cap \Omega_b}  | \nabla  u  |^2 (x) dx 
          \leq \sum_{i \in I} \int_{ B_i \cap \Omega_b } | \nabla  u  |^2 (x) dx 
          \leq \sum_{i \in I} \int_{ 6 B_i \cap \Omega_b } | \nabla  u  |^2 (x) dx . 
\end{equation}
To establish an estimate on the second term in~\eqref{e:grad}, we start by observing that, if $x \in \Omega_b $, then 
  \[
 \tilde u (x) : =     u(x)\theta_0(x) + \sum_{j \in I} m_i \theta_j (x). 
 \]
Next, we fix $i_0 $ in such a way that $ x \in B_{i_0}$ and we observe that, 
since $\nabla \theta_0+\sum_{j \in I}\nabla \theta_j =0$, we have 
\begin{eqnarray*}
\nabla \tilde u(x)&=&\theta_0(x)\nabla u(x)+\nabla \theta_0(x) u(x)+\sum_{j \in J_{i_0}}m_j\nabla \theta_j (x) \\
&=& \theta_0(x)\nabla u(x)+\underbrace{(u(x)-\widehat{m}_{i_0})\nabla \theta_0(x) }_{f_1}+\underbrace{\sum_{j \in J_{i_0}}(m_j -\widehat{m}_{i_0}) \nabla \theta_j(x)}_{f_2} .
\end{eqnarray*}
Next, we point out that, for any $i=1, \dots, \sharp I$ we have   
\begin{equation}
\label{e:fz}
\int_{2B_{i}\cap \Omega_b}| \theta_0(x)\nabla u(x)|^2 dx \leq \int_{2B_{i}\cap \Omega_b}|\nabla u(x)|^2 dx.
\end{equation}
Also, we recall that  $ \mathrm{supp}(\theta_0)\cap 2B_{i_0}\subseteq \widehat D_{i_0}$ and that $ |\nabla \theta_{i_0}| \leq  C(N)/ \delta$. We then recall that the Poicar\'e-Sobolev constant of $\widehat{D}_{i_0}$ is bounded by $C(N) \delta$ and by combining these  observations we get
\begin{equation}
\label{e:f11} 
\begin{split}  
\int_{2B_{i_0}\cap \Omega_b}|f_1(x)|^2& \leq \frac{C(N)}{\delta^2}\int_{\widehat{D}_{i_0}}|u(x)-\widehat{m}_{i_0}|^2 
\leq C(N) \int_{\widehat{D}_{i_0}}|\nabla u|^2 \\ &
\leq C(N) \int_{6B_{i_0}\cap \Omega_b}|\nabla u|^2 .  \\
\end{split}
\end{equation}
Finally, by recalling~\eqref{mimx} we have 
\begin{equation}
\label{e:f21}
\int_{2B_{i_0} \cap \Omega_b} |f_2|^2dx \leq  C(N) \int_{2B_{i_0} \cap \Omega_b} \sum_{j \in J_{i_0}}\frac{1}{\delta^2}(m_j -\widehat{m}_{i_0})^2\leq C(N) \int_{6B_{i_0}\cap \Omega_b}|\nabla u|^2 dx
\end{equation}
and by combining~\eqref{e:fz},~\eqref{e:f11} and~\eqref{e:f21} we infer   
\begin{equation}
\label{gradchi2}
         \int_{\Omega_a \cap \Omega_b \cap W} | \nabla \tilde u  |^2 (x) dx \leq 
         C(N) \sum_{i \in I} 
         \int_{6B_{i} \cap \Omega_b} |\nabla u|^2 dx.
         \end{equation}
To provide a bound on the third term in~\eqref{e:grad}, we observe that, if $x \in \Omega_a \setminus \Omega_b \subseteq \displaystyle{\bigcup_{i \in I} B_i}$, then 
 \[
 \tilde u (x) =      \sum_{i \in I} m_i \theta_i (x)  \qquad \theta_0(x) =0.
 \]
Hence, if we choose $i_0$ in such a way that $x \in B_{i_0}$, we get
$$
    \nabla \tilde u(x)=\sum_{j \in J_{i_0}}m_j \nabla \theta_j (x)= \sum_{j \in J_{i_0}}(m_j -\widehat{m}_{i_0})\nabla \theta_j (x).
$$ 
By arguing as in~\eqref{e:f21} we get 
$$
   \int_{B_{i_0}\setminus  \Omega_b} |\nabla \tilde u|^2 dx \leq C (N) \int_{6B_{i_0}\cap \Omega_b} |\nabla u|^2 dx,
$$
which implies 
\begin{equation}
\label{e:out}
    \int_{\Omega_a \setminus \Omega_b} |\nabla \tilde u|^2 (x) dx \leq 
    \sum_{i \in I} \int_{B_i \setminus \Omega_b} |\nabla \tilde u|^2 (x) dx \leq 
    C (N) \sum_{i \in I} \int_{6B_{i}\cap \Omega_b} |\nabla u|^2 dx. 
\end{equation}
Finally, by combining~\eqref{e:grad},~\eqref{e:fz},~\eqref{gradchi2} and~\eqref{e:out} we conclude that
\begin{equation}
\label{e:finalchi}
           \| {\bf 1}_{\Omega_a}  \nabla \tilde u - {\bf 1}_{\Omega_b}  \nabla u \|^2_{L^2 (\R^N)}  \leq  C (N) \sum_{i \in I} \int_{6B_{i}\cap \Omega_b} |\nabla u|^2 dx. 
\end{equation}
}
 {\sc $\diamond$ Step 6.} We conclude the proof of the Proposition by relying on Propositions~\ref{uestim} and~\ref{propdecayneum}. 
 
First, by combining Proposition~\ref{uestim} and~\eqref{e:chil2} we get
\begin{equation}
\label{e:final2}
\begin{split}
       \| \tilde u  {\bold 1}_{\Omega_a} - u {\bold 1}_{\Omega_b} 
         \|^2_{L^2 (\R^N)} & \leq 
         C(N) \sum_{i \in I} \int_{2 B_i \cap \Omega_b} u^2 (x) dx 
         \leq C(N) \sum_{i \in I}  \delta^N \| u\|^2_{L^{\infty} (\Omega_b)} \\
         & \leq 
          C(N, r_0, \diam(\Omega_b)) 
         \delta^N (1 + \sqrt{\mu} )^{2 \gamma(N)}
          \| u\|^2_{L^2 (\Omega_b)} \sharp I \\
         & \leq  C(N, r_0, \diam(\Omega_b)) L
         (1 + \sqrt{\mu} )^{2 \gamma(N)}\delta
          \| u\|^2_{L^2 (\Omega_b)} .   
          \end{split}
\end{equation} 
Next, we combine Proposition~\ref{propdecayneum} with~\eqref{e:finalchi} and we obtain
\begin{equation}
\label{e:finalg}
\begin{split}
       \| {\bf 1}_{\Omega_a}  \nabla \tilde u - {\bf 1}_{\Omega_b}  \nabla u \|^2_{L^2 (\R^N)} & \leq 
         C \mu (1 + \sqrt \mu )^{2 \gamma} \delta^{N-\eta} \| u \|^2_{L^2(\Omega_b)}
         \sharp I
         \leq 
         C L 
         \mu (1 + \sqrt \mu )^{2 \gamma} \delta^{\alpha} \| u \|^2_{L^2(\Omega_b)},
          \end{split}
\end{equation} 
provided that $\alpha= 1 - \eta$. In the previous expression, $C = C (N, r_0, \alpha, \diam(\Omega_b))$. By combining~\eqref{e:final2} and~\eqref{e:finalg} we conclude the proof. 
\end{proof}

\subsection{Conclusion of the proof of Theorem~\ref{main2}}
We finally conclude the proof of Theorem \ref{main2}.
\begin{proof}[Proof of Theorem \ref{main2}]  
By comparing~\eqref{e:projneu} with Proposition~\ref{estimproj2} and by arguing as in~\eqref{e:space} we get that the hypotheses of Lemma~\ref{l:ab} are satisfied provided that 
${A= B= n C (1 + \mu) ^{2 \gamma (N) +2}L \delta^\alpha}$, and hence by repeating the same argument as in the Dirichlet case we conclude. 
\end{proof}
\section{Acknowledgements} 

The authors wish to thank Dorin Bucur and Giuseppe Buttazzo for interesting conversations, and Michiel van den Berg for the reference quoted in Proposition \ref{propinfty1}. Also, L. V. Spinolo wishes to thank Giuseppe Savar\'e for having stimulated her interest in the topic, and for many useful discussions.  

E. Milakis was supported by the Marie Curie International Reintegration Grant No 256481 within the 7th European Community Framework Programme. Part of this work was done while A. Lemenant and L. V. Spinolo were both affiliated to the E. De Giorgi Research Center, Scuola Normale Superiore, Pisa, Italy, and when L.V. Spinolo was affiliated to the University of Zurich, Switzerland. Finally, L. V. Spinolo wishes to thank the Universit\'e Paris 7 and the Laboratoire J.L. Lions for supporting her visit, during which part of this work was done.
\bibliographystyle{plain}
\bibliography{biblio_eigen}



\begin{tabular}{l}
Antoine Lemenant\\
Universit\'e Paris Diderot - Paris 7 - LJLL - CNRS \\
U.F.R de Math\'ematiques \\
Site Chevaleret Case 7012\\
75205 Paris Cedex 13 FRANCE\\
{e-mail : \small \tt lemenant@ljll.univ-paris-diderot.fr}
\end{tabular}
\vspace{2em}

\begin{tabular}{l}
Emmanouil Milakis\\ 
University of Cyprus \\ 
Department of Mathematics \& Statistics \\ 
P.O. Box 20537\\
Nicosia, CY- 1678 CYPRUS\\
 {e-mail : \small \tt emilakis@ucy.ac.cy}
\end{tabular}
\vspace{2em}

\begin{tabular}{l}
Laura V. Spinolo\\
IMATI-CNR, \\
via Ferrata 1 \\
I-27100, Pavia, ITALY  \\
{e-mail : \small \tt spinolo@imati.cnr.it}
\hfill
\end{tabular}

\end{document}